\documentclass[a4paper]{amsart}

\usepackage{amsmath}
\usepackage{amssymb}
\usepackage{amsfonts}
\usepackage{mathrsfs}
\usepackage[OT2,T1]{fontenc}
\usepackage{array}
\usepackage{enumerate}
\usepackage[french, english]{babel}

\numberwithin{equation}{section}

\DeclareSymbolFont{cyrletters}{OT2}{wncyr}{m}{n}
\DeclareMathSymbol{\Sha}{\mathalpha}{cyrletters}{"58}

\newtheorem{thm}{Theorem}[section]
\newtheorem{prop}[thm]{Proposition}
\newtheorem{lem}[thm]{Lemma}
\newtheorem{cor}[thm]{Corollary}
\newtheorem{conj}[thm]{Conjecture}

\newtheorem{remark}[thm]{Remark}

\newtheorem*{thm*}{Theorem}
\newtheorem*{prop*}{Proposition}
\newtheorem*{lem*}{Lemma}
\newtheorem*{cor*}{Corollary}
\newtheorem*{hyp*}{Hypothesis}

\renewcommand{\bold}{\boldsymbol}
\newcommand{\mf}{\mathfrak}

\newcommand{\e}{\varepsilon}

\newcommand{\mc}{\mathcal}
\newcommand{\mb}{\mathbb}

\renewcommand{\i}{\infty}

\newcommand{\ovl}{\overline}

\newcommand{\mto}{\mapsto}

\newcommand{\xto}{\xrightarrow}

\renewcommand{\O}{\Omega_t}

\newcommand{\Q}{\mb{Q}}
\newcommand{\Z}{\mb{Z}}

\newcommand{\C}{\mb{C}}

\renewcommand{\o}{\mc{O}}

\DeclareMathOperator{\Gal}{Gal}

\DeclareMathOperator{\Sel}{Sel}
\DeclareMathOperator{\rank}{rank}

\DeclareMathOperator{\ord}{ord}

\DeclareMathOperator{\End}{End}

\DeclareSymbolFont{cyrletters}{OT2}{wncyr}{m}{n}
\DeclareMathSymbol{\Sha}{\mathalpha}{cyrletters}{"58}

\begin{document}

\title[Tamagawa number divisibility of $L$-values of twists of Fermat curve]{Tamagawa number divisibility of central $L$-values of twists of the Fermat elliptic curve}

\author{Yukako Kezuka}
\thanks{This work was supported by the SFB 1085 ``Higher invariants'' at the University of Regensburg, funded by the Deutsche Forschungsgemeinschaft (DFG)}

\maketitle
\selectlanguage{french} 
\begin{abstract}\'{E}tant donn\'{e} un entier $ N> 1 $ premier \`{a} $ 3 $, on d\'{e}signe par $ C_N $ la courbe elliptique $ x ^ 3 + y ^ 3 = N $. On \'{e}tudie d'abord la valuation $3$-adique  de la partie alg\'{e}brique de la valeur en $s=1$ de la fonction $ L $ de Hasse--Weil $L(C_N, s)$ de $C_N$ sur $\Q$ et on \'{e}tablit une relation entre la partie de $3$-torsion de son groupe de Tate--Shafarevich et le nombre de diviseurs premiers distincts de $ N $ qui sont inertes dans le corps quadratique imaginaire $ K=\Q(\sqrt {-3})$.  Dans la cas où $L(C_N,1)\neq 0$ et $N$ est en produit de nombres premiers décomposés dans $K$, on montrer que l'ordre du groupe de Tate--Shafarevich tel que prédit par la conjecture de la conjecture de Birch et Swinnerton-Dyer est un carré parfait.

\end{abstract}
\selectlanguage{english}
\begin{abstract} Given any integer $N>1$ prime to $3$, we denote by $C_N$ the elliptic curve $x^3+y^3=N$. We first study the $3$-adic valuation of the algebraic part of the value of the Hasse--Weil $L$-function $L(C_N,s)$ of $C_N$ over $\mathbb{Q}$ at $s=1$, and we exhibit a relation between the $3$-part of its Tate--Shafarevich group and the number of distinct prime divisors of $N$ which are inert in the imaginary quadratic field $K=\mathbb{Q}(\sqrt{-3})$. In the case where $L(C_N,1)\neq 0$ and $N$ is a product of split primes in $K$, we show that the order of the Tate--Shafarevich group as predicted by the conjecture of Birch and Swinnerton-Dyer is a perfect square.
\end{abstract}

\section{Introduction and main results}
The study of elliptic curves dates back at least to the $17$th century when Fermat proved that there is no right-angled triangle whose side lengths are integers and whose area is a perfect square. Much of the arithmetic of elliptic curves remains mysterious and has been one of the main focuses of number theory in the past century. This is highlighted by the remarkable conjecture of Birch and Swinnerton-Dyer, which was developed with the help of numerical calculations \cite{bsd2}. Given an elliptic curve $E$ defined over a number field $F$, the conjecture relates the behaviour of its Hasse--Weil {$L$-function} $L(E/F,s)$ (in the complex variable $s$, convergent in the half plane $\mathrm{Re}(s)>\frac{3}{2}$) at $s=1$ to the order of its Tate--Shafarevich group $\Sha(E/F)$ (not known to be finite in general). The Tate--Shafarevich group is defined by
$$\Sha(E/F)=\ker\left(H^1(F,E(\ovl{F}))\to \prod_v H^1(F_v, E(\ovl{F}_v))\right),$$
where $v$ runs over all valuations of $F$, thus its order measures the failure of the local-to-global principle for $E/F$. Thankfully, in the two cases when:
\begin{enumerate}
\item $\End_F(E)\neq \Z$, or
\item $E$ is defined over $\Q$,
\end{enumerate}
we now know that $L(E/F,s)$ has analytic continuation to the whole complex plane and satisfies a functional equation relating its value at $s$ with its value at $2-s$. In particular, $L(E/F,s)$ is known to be defined at $s=1$ in these cases, and its value is called the central value. Case (1) follows from a theorem of Deuring, which allows us to identify $L(E/F,s)$ with a product of Hecke $L$-functions with Grossencharacter, whereas case (2) is a consequence of the modularity theorem. If an elliptic curve $E/F$ satisfies (1), we say that $E$ has complex multiplication. In this case, we have a canonical injection
$$K=\Q\otimes_\Z \End_F(E)\hookrightarrow F$$
given by the action of $\End_F(E)$ on the space of holomorphic $F$-differentials on $E$, and we also know that $K$ is an imaginary quadratic field. Suppose in addition that $L(E/F,1)\neq 0$. Iwasawa theory is an extremely powerful tool in studying, for a prime number $p$, the $p$-part of the Birch--Swinnerton-Dyer conjecture, which predicts the order of the $p$-primary subgroup $\Sha(E/F)(p)$ of $\Sha(E/F)$. If $F=K$, Rubin used Iwasawa theory to show in \cite{rubin} that the $p$-part of the Birch--Swinnerton-Dyer conjecture holds for $E$ for all $p$ not dividing the order of the group of roots of unity in $K$. In a forthcoming paper with  J.~Coates, Y.~Li and Y.~Tian, we shall show, using Iwasawa theory, the $p$-part of the Birch--Swinnerton-Dyer conjecture for a family of elliptic curves $E/F$ where $F$ is the Hilbert class field of $K$ (in which $2$ splits) for any prime $p$ which splits in $K$. This paper, however, will only deal with elliptic curves with complex multiplication defined over $\Q$, and henceforth we will omit $F$ from the notation.\\

Iwasawa theory emerged as a study of the growth of of arithmetic objects in towers of field extensions which are unramified outside of $p$. We will take an alternative standpoint in this paper, and investigate the ``horizontal'' growth of arithmetic objects as the number of ramified primes grows, while keeping the degree of the field extensions fixed. Let $N>1$ be a cube-free integer prime to $3$, and let $\sqrt[3]{N}$ denote the real cube root. We will study a family of twists of the Fermat elliptic curve $X^3+Y^3=Z^3$ by the cubic extension $\Q(\sqrt[3]{N})/\Q$ whose equation in terms of the affine plane is given by
$$C_{N}: x^3+y^3=N.$$
They have complex multiplication by the ring of integers $\o$ of the imaginary quadratic field $K=\Q(\sqrt{-3})$. In the case when $N=2p^i$ with $i\in \{0,1\}$ and $p\equiv 2,5\bmod 9$, an odd prime number, the $3$-part of the Birch--Swinnerton-Dyer conjecture has recently been established in joint work with Y.~Li \cite{kl}. This was achieved by expressing Hecke $L$-functions in terms of Eisenstein series, studying congruences between $L$-values to obtain the ``explicit modulo $3$ Birch--Swinnerton-Dyer conjecture'', and combining with a result of Cai--Shu--Tian \cite{CST} which uses an explicit Gross--Zagier formula.  In \cite{kl}, we also related the $2$-part of the Tate--Shafarevich group $\Sha(C_{2p^i})[2]$ of $C_{2p^i}$ over $\Q$ to the $2$-part of the ideal class group $\mathrm{Cl}(L)[2]$ of $L=\Q(\sqrt[3]{p})$, following the descent argument of C.~Li \cite{cl}. This was used to find examples of elliptic curves of the form $C_{2p^i}$ with rank $1$ and non-trivial $\Sha(C_{2p^i})[2]$. In particular, it became apparent that the $2$-part of the Birch--Swinnerton-Dyer conjecture is much deeper in nature for these curves, coming from the fact that $\Sha(C_{2p^i})[2]$ is not necessarily trivial.\\ 

In this paper, we will study the $2$-adic and $3$-adic valuation of the algebraic part $L^{(\mathrm{alg})}(C_N,1)\in \Q$ of the central $L$-value when it does not vanish, in the case when $N$ is divisible by an arbitrary number of distinct primes, split or inert in $K$. It is well known that $C_N(\Q)_{\mathrm{tor}}$ is trivial for any cube-free integer $N>2$. Furthermore, we know from the theorem of Rubin \cite{rubin} that $\Sha(C_N)$ is finite when $L(C_N,1)\neq 0$. Thus the $p$-part of the Birch--Swinnerton-Dyer conjecture predicts the following:

\begin{conj}\label{pBSD}Let $N>2$ a cube-free integer, and assume $L(C_N,1)\neq 0$. Then for any prime number $p$, we have
$$\ord_p\left(L^{(\mathrm{alg})}(C_N,1)\right)=\ord_p\left(\#(\Sha(C_N)(p))\right)+\ord_p\left(\prod_{q \:\mathrm{ bad}}c_q\right),$$
where the product runs over the primes $q$ of bad reduction for $C_N$, and $c_q=[C_N(\Q_q): C_N^0(\Q_q)]$ is the Tamagawa factor of $C_N$ at $q$. Here, $C_N^0(\Q_q)$ denotes the subgroup of points with non-singular reduction modulo $q$.
\end{conj} 

\begin{remark} We repeat that the above is known for $p>3$ from \cite{rubin}, since the group of roots of unity $\bold{\mu}_K$ of $K=\Q(\sqrt{-3})$ has order $6$. Thus it remains to show this for $p=2$ and $p=3$.
\end{remark}

We will study the following three families of curves separately:
\begin{align*}
E_{D}&: y^2=x^3-2^4 3^3D^2\\
E_{2D}&: y^2=x^3-3^3D^2\\
E_{4D}&: y^2=x^3-2^23^3D^2,
\end{align*}
where $D$ is a cube-free product of $n$ distinct primes with $(D,6)=1$. We see that $E_{N}$ has $C_{N}$ as a birationally equivalent form.\\

 In Section \ref{section2} of this paper, we will show the following:
 
 \begin{thm}\label{m1} Let $N>1$ be a cube-free integer with $(N,3)=1$. Let $k(N)$ be the number of distinct prime factors of $N$. Then we have
 \begin{align*}\ord_3\left(L^{(\mathrm{alg})}(C_{N},1)\right)\geq \left\{ \begin{array}{ll}
 k(N)&\mbox{if $N$ is a product of split primes in $K$} \\
 k(N)-1&\mbox{otherwise.}\\
       \end{array} \right.
 \end{align*}
 \end{thm}
 
  \begin{remark}When $N$ is a product of split primes of the form $p\o=(\pi)(\ovl{\pi})$ with $\pi$ congruent to $1$ modulo $6$, a lower bound of $k(N)-1$ has already been obtained in \cite{qiu}. This was improved to $k(N)+1$ for $N$ a product of ``cubic-special'' primes, which occupy density $\frac{2}{3}$ of primes congruent to $1$ modulo $27$. See \cite{yuka} for more details. In this paper, we will follow the methods in \cite{yuka}, which is an adaptation to cubic twists of $C_1\colon x^3+y^3=1$ of Zhao's method \cite{zhao} for quadratic twists of the curves $y^2=x^3-x$. 
 \end{remark}
 
 \begin{remark} Given a cube-free integer $N$ with $(N,3)=1$, let $r(N)$ and $s(N)$ be the number of distinct prime factors dividing $N$ which are inert and split in $K$, respectively. In particular, we have $k(N)=r(N)+s(N)$. A computation of Tamagawa numbers of $C_N$ (see Lemma \ref{tam}) gives 
  \begin{align*}\ord_3\left(\prod_{q \mathrm{ bad}}c_q\right) = \left\{ \begin{array}{ll}
s(N)+1&\mbox{if $N\equiv 1,8 \bmod 9$} \\
 s(N)&\mbox{otherwise.}
       \end{array} \right.
 \end{align*}
Let $t(N)=1$ if $N\equiv \pm 1\bmod 9$ and $t(N)=0$ if $N\equiv \pm 2$ or $\pm 4 \bmod 9$. To summarise, Conjecture \ref{pBSD} predicts that when $L(C_N,1)\neq 0$,  
   \begin{enumerate}
  \item $\Sha(C_N)[3]$ is non-trivial when $r(N)\geq 2+t(N)$.
  \item $\Sha(C_N)[3]$ is trivial when $r(N)\leq 1+ t(N)$ and equality holds in Theorem \ref{m1}.
 \end{enumerate}
Furthermore, by a root number computation (see Proposition \ref{lroot}), we see that the sign of the global root number $\epsilon(C_N/\Q)$ depends on the parity of $r(N)$ and $t(N)$. To be more precise, we have
 $$\epsilon(C_N/\Q)=-(-1)^{t(N)}(-1)^{r(N)}.$$
 \end{remark}
 
\vspace{10pt}
 
We check in Section \ref{section2a} that the above prediction is indeed consistent with a bound given by a $3$-descent:
\begin{thm}\label{3descent} Let $C_{N}: x^3+y^3=N$, where $N> 1$ is a cube-free integer. Then the $3$-part of the Tate--Shafarevich group $\Sha(C_{N})[3]$ of $C_{N}$ over $\Q$ satisfies
$$\ord_3\left(\#\left(\Sha(C_{N})[3]\right) \right)\geq r(N)-t(N)-1-\rank(C_N),$$
where $r(N)$ is the number of distinct primes dividing $N$ which are inert in $\Q(\sqrt{-3})$, $t(N)=1$ if $N\equiv \pm 1 \bmod 9$, $t(N)=-1$ if $\ord_3(N)=1$ and $t(N)=0$ otherwise. 
\end{thm}

Note that it is possible to have non-trivial $\Sha(C_N)[3]$ when $r(N)\leq 1+t(N)$. When $N= 2\cdot 5\cdot 1657$, for example, we have $r(N)=2=1+t(N)$ and $\Sha(C_N)[3]=(\Z/3\Z)^2$. More numerical examples, obtained using Magma \cite{magma} and Sage \cite{sage}, are listed in Appendix \ref{appendixB}. This curious relation between $\Sha(C_N)[3]$ and the number of distinct inert prime divisors of $N$ will be studied in more detail in a subsequent paper.\\

We next study the invariant 
$$\mathcal{S}_N=\frac{L^{(\mathrm{alg})}(C_N,1)}{\prod_{q \mathrm{ bad}}c_q}.$$ 
This is defined so that if $L(C_N,1)\neq 0$, the full Birch--Swinnerton-Dyer conjecture predicts 
$$\mathcal{S}_N=\#(\Sha(C_N))$$
for $N>2$. Rosu showed in \cite{rosu} that if $N$ is a cube-free product of split primes, $\mathcal{S}_N$ is a perfect square up to a denominator of an even power of $3$, and gives a formula for its value in terms of theta functions. The Tamagawa number divisibility in Theorem \ref{m1} gives the following strengthening of \cite[Theorem 1.1--Theorem 1.3]{rosu}.

\begin{cor}\label{corro}Let $N>2$ be a cube-free positive integer with $(N,3)=1$. Then 
\begin{enumerate}
\item $\mathcal{S}_N$ is a perfect square if $N$ is a product of split primes.
\item $2\mathcal{S}_N\in \Z$ if $N\equiv 2,7\bmod 9$.
\item $\mathcal{S}_N\in \Z$ otherwise.
\end{enumerate} 
Furthermore, we have the trace formula
$$\mathcal{S}_N=\frac{1}{3\prod_{q \mathrm{ bad}}c_q}\mathrm{Tr}_{K(j(\o_{3N}))/K}\left(\sqrt[3]{N}\frac{\Theta_K(N\omega)}{\Theta_K(\omega)}\right)$$
where $\Theta_K(z)=\sum_{a,b\in \Z}e^{2\pi i z (a^2+b^2-ab)}$ is the theta function of weight one associated to $K$, $\omega=\frac{-1+\sqrt{-3}}{2}$ is a cube root of unity and $K(j(\o_{3N}))$ is the ring class field of the order $\o_{3N}=\Z+3N\o$.
\end{cor}

When $N$ is a product of split primes, a more detailed description of $\mathcal{S}_N$ in terms of theta functions of weight $1/2$ can be found in \cite[Theorem 1.2]{rosu}.\\

In Section \ref{section3}, we study the $2$-adic valuation of $L^{(\mathrm{alg})}(C_N,1)$, which readily yields the following when combined with Theorem \ref{m1} and the $p$-part of the Birch--Swinnerton-Dyer conjecture known for $p>3$.
 
\begin{thm}\label{m2}
Let $N>2$ be a cube-free integer with $(N,3)=1$. Then we have
$$L^{(\mathrm{alg})}(C_{N}, 1)\in \Z.$$
\end{thm}
We note that this is already known from \cite[Theorem 1]{stephens} which uses class field theory and the methods of \cite{bsd2}, and Corollary \ref{corro} even shows that $\frac{L^{(\mathrm{alg})}(C_{N}, 1)}{\prod_{q \mathrm{ bad}}c_q} \in \Z$ when $N\not\equiv 2,7\bmod 9$. But we will provide another proof by directly studying the $2$-adic valuation, which is expected to be useful in showing the Tamagawa number divisibility for quadratic or sextic twists of the curve $y^2=x^3-2^43^4$. It also gives an insight into how $L^{(\mathrm{alg})}(C_{N}, 1)$ can be divisible by a large power of $2$ in the case of quadratic or sextic twists.\\

Finally, if the central $L$-value does not vanish, we expect that Iwasawa theory can be used to study the $2$-part of the Birch--Swinnerton-Dyer conjecture for the curves $C_{2p^i}$ studied in \cite{kl}. This would give us the full Birch--Swinnerton-Dyer conjecture for $C_{2p^i}$.

\subsection*{Acknowledgement} 
I would like to thank Ashay Burungale for insightful conversations on the topics of this paper and drawing my attention to the work of Rosu. I would also like to thank the referee for carefully reading the manuscript and giving valuable comments and suggestions. Furthermore, I am grateful to the Max Planck Institute for Mathematics in Bonn for its support and hospitality. 

\section{The $3$-adic valuation of the algebraic part of central $L$-values}\label{section2}

Given a cube-free integer $N>1$ with $(N,3)$, we define an elliptic curve
$$C_N: x^3+y^3=N.$$
Then $C_N$ has complex multiplication by the ring of integers $\o$ of the imaginary quadratic field $K=\Q(\sqrt{-3})$. From now on, we will write $N=tD$ where $t=1, 2$ or $4$ and $2\nmid D$. Suppose $D$ has $n$ distinct prime factors, say $p_1,\ldots, p_n$. Given any $\alpha=(\alpha_1,\ldots , \alpha_n)\in \{0,1, 2\}^n$, we define a cube-free product $D_\alpha=p_1^{\alpha_1}\cdots p_n^{\alpha_n}$. If $p_i$ is inert in $K$, then we have $p_i\equiv 2\bmod 3$. If $p_i$ splits in $K$, say $p_i\o=\mf{p}_i\ovl{\mf{p}_i}$, then $p_i\equiv 1\bmod 3$, and we may choose a generator $\pi_i$ of $\mf{p}_i$ so that $\pi_i\equiv \ovl{\pi}_i\equiv 1\bmod 3\o$. We define $E_{tD_\alpha}$ by the equations
\begin{align}\label{eqE}
E_{D_\alpha} &\colon y^2=x^3- 2^4 3^3 D_\alpha^2 \nonumber\\
E_{2D_\alpha} &\colon y^2=x^3- 3^3 D_\alpha^2\\
E_{4D_\alpha} &\colon y^2=x^3- 2^2 3^3 D_\alpha^2 \nonumber.
\end{align}
Note that $E_{tD_\alpha}$ is birationally equivalent to $C_{tD_\alpha}$. Furthermore, we may consider $\alpha=(\alpha_1,\ldots ,\alpha_n)$ as an element of $(\Z/ 3 \Z)^n$, since $E_{tD_{\alpha'}}$ and $E_{tD_{\alpha}}$ are isomorphic over $\Q$ if $\alpha'\equiv \alpha\bmod 3$. Let $A$ be the elliptic curve given by the classical Weierstrass equation
$$A: y^2=4x^3-1,$$
and let $\Omega=3.059908074\ldots $ denote the real period associated to the differential $\frac{\mathrm{d}x}{y}$ on $A$. We set $\Omega_N=\frac{\Omega}{\sqrt{3}\sqrt[3]{N}}$, and let $\mc{L}_N=\Omega_N \o$ denote the corresponding period lattice. In particular, we have $\Omega_1=1.766638750\ldots$ and $\Omega_t=\frac{\Omega_1}{\sqrt[3]{t}}$. We define the algebraic part of $L(C_N,1)$ by
\begin{equation}\label{periods}L^{(\mathrm{alg})}(C_N,1)=\frac{L(C_N,1)}{\Omega_N},
\end{equation}
which is well known to be a rational number (see, for example, \cite{gol-sch}).  We write $\psi_{N}$ for the Grossencharacter of $C_N/K$, and fix an embedding $K\hookrightarrow \C$. Since $C_N$ is defined over $\Q$, we may identify $L(C_N,s)$ with the Hecke $L$-function $L(\ovl{\psi}_{N},s)$ where $\ovl{\psi}_{N}$ denotes the complex conjugate of $\psi_N$ (see, for example, \cite[Theorem 45]{coates1}). Given an integer $m$, the radical $\mathrm{rad}(m)$ of $m$  is the product of distinct primes dividing $m$. Then the conductor of $\psi_{t}$ divides $3\mathrm{rad}(t) \o$, and the conductor $\psi_{tD}$  divides $\mf{f}=(f)$ where $f=3\mathrm{rad}(tD)$ (see \cite{stephens}). Let $C_{tD_\alpha,\mf{f}}=\ker\left(C_{tD_\alpha}\xto{[f]}C_{tD_\alpha}\right)$. Then by \cite[Proposition 48]{coates1}, $K(C_{tD_\alpha,\mf{f}})$ coincides with the ray class field $K(\mf{f})$ of $K$ modulo $\mf{f}$.  Writing $M=f/3=\mathrm{rad}(tD)$ and $\widetilde{\bold{\mu}}_6$ for the image of the group of roots of unity in $K$ under reduction modulo $\mf{f}$, we have 
\begin{align*}
 (\o/M\o)^\times&\simeq (\o/f\o)^\times /\widetilde{\bold{\mu}}_6\\
&\simeq \Gal(K(\mf{f})/K),
\end{align*}
where the first isomorphism will be fixed later, and the second isomorphism is given by the Artin map. Thus, writing $\Delta=\mathrm{rad}(D)=p_1\cdots p_n$, the Chinese remainder theorem gives
$$
\Gal(K(\mf{f})/K) \simeq \left\{ \begin{array}{ll}
 (\o/\Delta\o)^\times &\mbox{if $t=1$} \\
 (\o/2\o)^\times \times (\o/\Delta\o)^\times &\mbox{otherwise.}\\
       \end{array} \right.
$$

 Let $S$ be the set of integral primes in $K$ dividing $f$. Given any $\alpha\in (\Z/3\Z)^n$, we denote by $L_S(\ovl{\psi}_{tD_\alpha},s)$ the (usually imprimitive) Hecke $L$-function obtained by omitting the Euler factors at the primes $v$ in $S$. We will first express the Hecke $L$-functions in terms of a sum of Eisenstein series following the results of Goldstein and Schappacher \cite{gol-sch}. Recall that the Kronecker--Eisenstein series in complex variables $z$ and $s$ is defined, for any complex lattice $\mc{L}$ and any integer $k\geq 1,$ by
$$H_k(z,s,\mc{L})=\sum_{w\in \mc{L}}\frac{(\ovl{z}+\ovl{w})^k}{\;|z+w|^{2s}},$$
where the sum is taken over $w$ in $\mc{L}$, except $-z$ if $z\in \mc{L}$. The series converges in the half plane $\mathrm{Re}(s)>1+\frac{k}{2}$, and it has analytic continuation to the whole complex $s$-plane (see \cite[Theorem 1.1]{gol-sch}). The non-holomorphic Eisenstein series $\mc{E}_1^*(z,\mc{L})$ is defined by
$$\mc{E}_1^*(z,\mc{L})=H_1(z,1,\mc{L}).$$
We also recall that for $a\in K^*$ and an ideal $\mf{b}$ of $K$ coprime to $3a$, the cubic residue symbol $\left(\frac{a}{\mf{b}}\right)_3$ is defined by the equation 
$$(\sqrt[3]{a})^{\tau_\mf{b}}=\left(\frac{a}{\mf{b}}\right)_3 \sqrt[3]{a},$$
where $\tau_\mf{b}=(\mf{b}, K(\sqrt[3]{a})/K)$ denotes the Artin symbol of $\mf{b}$ in $\Gal(K(\sqrt[3]{a})/K)$. Given two elements $a, b\in K^*$ with $b$ coprime to $3a$, we also write
$$\left(\frac{a}{b}\right)_3=\left(\frac{a}{(b)}\right)_3,$$
and we have $\left(\frac{a}{b}\right)_3=\left(\frac{b}{a}\right)_3$ if in addition $a, b\equiv \pm 1\bmod 3\o$.

Let $\mc{B}$ be any set of integral primes of $K$ prime to $\mf{f}$ such that the Artin symbol $\sigma_\mf{b}$ of $\mf{b}$ in $K(\mf{f})/K$ runs over the Galois group $\Gal(K(\mf{f})/K)$ precisely once as $\mf{b}$ runs over $\mc{B}$. Then by \cite[Proposition 5,5]{gol-sch}, we have
\begin{align*}\frac{L_S(\ovl{\psi}_{tD_\alpha},1)}{\Omega_{tD_\alpha}}&=\frac{1}{f}\sum_{\mf{b}\in \mc{B}}\mc{E}^*_1\left(\frac{\Omega_{tD_\alpha}}{f}, \mc{L}_{tD_\alpha}\right)^{\sigma_\mf{b}}.
\end{align*}

Now, we have $\Omega_{tD_\alpha}=\frac{\Omega_t}{\sqrt[3]{D_\alpha}}$ and $\mathcal{E}^*_1$ is homogeneous of degree $-1$. It follows that for any $\alpha\in (\Z/3 \Z)^n$, we have
\begin{align*}
\frac{L_S(\ovl{\psi}_{tD_\alpha},1)}{\Omega_t}&=\frac{1}{f}\sum_{\mf{b}\in \mc{B}}(D_\alpha^2)^\frac{\sigma_\mf{b}-1}{6}\mc{E}^*_1\left(\frac{\Omega_t}{f}, \mc{L}_t\right)^{\sigma_\mf{b}}\\
&=\frac{1}{f}\sum_{\mf{b}\in \mc{B}}\left(\frac{D_\alpha}{\mf{b}}\right)_3\mc{E}^*_1\left(\frac{\Omega_t}{f}, \mc{L}_t\right)^{\sigma_\mf{b}}.
\end{align*}
Let $\mc{C}$ be a set of elements of $\o$ prime to $M$ such that $c\mod M$ runs over $(\o/M\o)^\times$ precisely once as $c$ runs over $\mc{C}$, and such that $c\in \mc{C}$  implies $-c\in \mc{C}$. This is possible because $c\equiv -c\bmod M$ if and only if $c\equiv 0\bmod D$, which is absurd.
We may now set
$$\mc{B}=\{(3c+\epsilon M) : c\in \mc{C}\},$$
where the sign $\epsilon\in \{\pm 1\}$ of $M$ is chosen so that $\epsilon M\equiv 1\bmod 3$. In particular, we have $3c+\epsilon M \equiv 1\bmod 3\o$ for all $c\in \mc{C}$.
We define
$$V=\{c\in \mc{C}: \left(\frac{p}{\mf{b}}\right)_3=1 \text{ for all $p\mid D$, where }\mf{b}=(3c+\epsilon M)\in \mc{B}\}.$$

Now, for all $p\mid D$ and $\mf{b}=(3c+\epsilon M)\in \mc{B}$, we have
\begin{align*}\left(\frac{p}{\mf{b}}\right)_3&=\left( \frac{3c+\epsilon M}{p}\right)_3 \text{ (since both are congruent to $\pm 1 \bmod 3\o$)}\\
&=\left( \frac{3c}{p}\right)_3 \text{ (since $p\mid M$)}\\
&=\left( \frac{c}{p}\right)_3.
\end{align*}
Thus we see that if $c\in V$ then $-c\in V$, as $(-1)^3=-1$ is always cubic residue. We write $r$ and $s$ for the number of inert and split primes in the prime factorisation of $\Delta=\mathrm{rad}(D)$, respectively. In particular, $n=r+s$.
\begin{lem}\label{lemV} The Artin symbols of $\mf{b}=(3c+\epsilon M)$ in $\Gal(K(\mf{f})/K)$ give a set of representatives of the Galois group $\Gal\left(K(\mf{f})/K(\sqrt[3]{p_1},\ldots ,\sqrt[3]{p_n})\right)$ as $c$ runs over $V$. Therefore, \mbox{$3^{r+2s+1-n}\mid \#(V)$} if $2\mid t$, and $3^{r+2s-n}\mid \#(V)$ if $t=1$.
\end{lem}

\begin{proof}
By the definition of the cubic residue symbol, we see that $\sigma_{\mf{b}}$ acts trivially on $\sqrt[3]{D_\alpha}$ for all $\alpha\in (\Z/3\Z)^n$ if and only if $\left(\frac{p}{\mf{b}}\right)_3=1$ for all primes $p\mid D$, which happens if and only if $c\in V$ where $c$ is such that $\mf{b}=(3c+\epsilon M)$. Thus the first claim follows. Now, 
$$\#\left(\Gal(K(\mf{f})/K)\right)=\#\left(\left(\o/M \o)\right)^\times  \right)=3^e\prod_{p\mid D \text{ inert}}(p^2-1)\prod_{p'\mid D \text{ split}}(p'-1)^2,$$ where $e=1$ if $2\mid t$ and $e=0$ if $t=1$. It follows that $\ord_3\left(\#\left(\Gal(K(\mf{f})/K)\right)\right)\geq r+2s+e$. Thus 
$$\ord_3\left(\#\left(\Gal\left(K(\mf{f})/K(\sqrt[3]{p_1},\ldots , \sqrt[3]{p_n})\right)\right)\right)\geq r+2s+e-n,$$
as required.
\end{proof}

We consider the following sum of (usually imprimitive) Hecke $L$-functions:
$$\Phi_{tD}=\sum_{\alpha\in (\Z/3\Z)^n}\frac{L_S(\ovl{\psi}_{tD_\alpha},1)}{\Omega_t}.$$
The subscript $S$ will often be omitted from the notation whenever it is primitive.  We will now study the $3$-adic valuation of this sum.

\begin{thm}\label{thmPhi} We have
\begin{enumerate}
\item $\ord_3\left(\Phi_{tD}\right)\geq n-1$ if $t=1$ and $n=r$.
\item $\ord_3\left(\Phi_{tD}\right)\geq n-\frac{1}{3}$ otherwise.
\end{enumerate}
\end{thm}

\begin{proof} 
Note that 
$\varphi^\mf{b}: \alpha\mto \left(\frac{D_\alpha}{\mf{b}}\right)_3$ gives a $1$-dimentional character of $(\Z/3\Z)^n$. Any $1$-dimentional character is irreducible, and $(\Z/3\Z)^n$ is an abelian group. Thus considering the inner product of $\varphi^\mf{b}$ and the trivial character $id$, we obtain
\begin{align*}\langle \varphi^\mf{b}, id\rangle &=\frac{1}{|(\Z/3\Z)^n)|}\sum_{\alpha\in (\Z/3\Z)^n}\varphi^\mf{b}(\alpha)\ovl{id(\alpha)}=\frac{1}{3^n}\sum_{\alpha\in (\Z/3\Z)^n}\left(\frac{D_\alpha}{\mf{b}}\right)_3\\
&=\left\{ \begin{array}{ll}
 1 &\mbox{if $\varphi^\mf{b}(\alpha)=1$ for all $\alpha\in (\Z/3\Z)^n$} \\
 0 &\mbox{otherwise.}
       \end{array} \right.
\end{align*}
 It follows that
$$\sum_{\alpha \in (\Z/3\Z)^n}\varphi^\mf{b}(\alpha)=\left\{ \begin{array}{ll}
 3^n &\mbox{if $c\in V$} \\
 0 &\mbox{otherwise.}
       \end{array} \right.$$
Hence, we have
\begin{align*}
\Phi_{tD}&=\frac{1}{f}\sum_{\mf{b}\in \mc{B}}\left(\sum_{\alpha\in (\Z/3\Z)^n} \left(\frac{D_\alpha}{\mf{b}}\right)_3 \right)\mc{E}^*_1\left(\frac{\Omega_t}{f},\mc{L}_t\right)^{\sigma_\mf{b}}\\
&=\frac{3^n}{f}\sum_{c\in V}\mc{E}^*_1\left(\frac{\Omega_t}{f},\mc{L}_t\right)^{\sigma_\mf{b}}.
\end{align*}
Given an integral ideal $\mf{a}$ of $K$, we have
$$\psi_{t}(\mf{a})=\ovl{\left(\frac{t}{\mf{a}}\right)}_3 a,$$
where $a$ is a generator of $\mf{a}$ satisfying $a\equiv 1\bmod 3$ (see, for example, \cite[Chapter II, Example 10.6]{silverman}). 
Thus using the fact that $\mc{E}^*_1\left(\frac{\Omega_t}{f},\mc{L}_t\right)^{\sigma_\mf{b}}=\mc{E}^*_1\left(\frac{\psi_t(\mf{b})\Omega_t}{f},\mc{L}_t\right)$ (see \cite{gol-sch}), we obtain
\begin{align*}
\Phi_{tD}=\frac{3^n}{f}\sum_{c\in V}\left(\frac{t}{\mf{b}}\right)_3\mc{E}^*_1\left(\frac{\epsilon \Omega_t}{3}+\frac{c\Omega_t}{M},\mc{L}_t\right).
\end{align*}
Since $\ord_3(f)=1$, it suffices to show  
$$
\ord_3\left(\sum_{c\in V}\left(\frac{t}{\mf{b}}\right)_3\mc{E}^*_1\left(\frac{\epsilon \Omega_t}{3}+\frac{c\Omega_t}{M},\mc{L}_t\right)\right)\geq\left\{ \begin{array}{ll}
0 &\mbox{if $t=1$ and $n=r=1$} \\
\frac{2}{3} &\mbox{otherwise.}
       \end{array} \right.$$
Define $s_2(\mathcal{L}_t)=\lim_{\substack{s\to 0 \\ s>0}}\sum_{w\in \mathcal{L}_t\backslash \{0\}}w^{-2}|w |^{-2s}$ and  $A(\mathcal{L}_t)=\frac{\ovl{u}v-u\ovl{v}}{2\pi i}$ where $(u,v)$ is a basis of $\mathcal{L}_t$ over $\mathbb{Z}$ satisfying $\mathrm{Im}(v/u)>0$. We write $\zeta(z,\mathcal{L}_t)$ for the Weierstrass zeta function of $\mathcal{L}_t$. Then by \cite[Proposition 1.5]{gol-sch}, we have
$$\mathcal{E}_1^*(z,\mathcal{L}_t)=\zeta(z,\mathcal{L}_t)-zs_2(\mathcal{L}_t)-\overline{z}A(\mathcal{L}_t)^{-1}.$$
Thus writing $\omega=\frac{-1+\sqrt{-3}}{2}$, we have $A(\mathcal{L}_t)=\frac{\Omega_t^2(\omega-\overline{\omega})}{2\pi i}=\frac{\sqrt{3}\Omega_t^2}{2\pi}$, and we can see that $s_2(\mathcal{L}_t)=0$ on noting that $\omega$ stabilises $\mathcal{L}_t$, which gives $\omega^{-2} s_2(\mathcal{L}_t)=s_2(\mathcal{L}_t)$. Hence
$$\mathcal{E}_1^*(z, \mathcal{L}_t)=\zeta(z,\mathcal{L}_t)-\frac{2\pi\overline{z}}{\sqrt{3}\Omega_t^2}.$$
Recall also that for $z_1, z_2\in \mathbb{C}$, we have an addition formula:
$$\zeta(z_1+z_2, \mathcal{L}_t)=\zeta(z_1, \mathcal{L}_t)+\zeta(z_2, \mathcal{L}_t)+\frac{1}{2}\frac{\wp'(z_1, \mathcal{L}_t)-\wp'(z_2, \mathcal{L}_t)}{\wp(z_1, \mathcal{L}_t)-\wp(z_2, \mathcal{L}_t)}.$$
Applying this with $z_1=\frac{\epsilon \Omega_t}{3}$, $z_2=\frac{c\Omega_t}{M}$, we get
\begin{SMALL}
\begin{align*}&\sum\limits_{c\in V}\left(\frac{t}{\mf{b}}\right)_3\mathcal{E}_1^*\left(\frac{c\Omega_t}{M}+\frac{\epsilon \Omega_t}{3}, \mathcal{L}_t\right)=\sum\limits_{c\in V}\left(\frac{t}{\mf{b}}\right)_3\left(\zeta\left(\frac{c\Omega_t}{M}+\frac{\epsilon \Omega_t}{3}, \mathcal{L}_t\right)-\left(\frac{\overline{c}\Omega_t}{M}+\frac{\epsilon \Omega_t}{3}\right)\frac{2\pi}{\sqrt{3}\Omega_t^2}\right)\\
&=\sum\limits_{c\in V}\left(\frac{t}{\mf{b}}\right)_3\left(\epsilon  \zeta\left(\frac{\Omega_t}{3}, \mathcal{L}_t\right)+\zeta\left(\frac{c\Omega_t}{M}, \mathcal{L}_t\right)+\frac{1}{2}\frac{\epsilon  \wp'(\frac{\Omega_t}{3}, \mathcal{L}_t)-\wp'(\frac{c\Omega_t}{M}, \mathcal{L}_t)}{\wp(\frac{\Omega_t}{3}, \mathcal{L}_t)-\wp(\frac{c\Omega_t}{M}, \mathcal{L}_t)}-\left(\frac{\overline{c}\Omega_t}{M}+\frac{\epsilon \Omega_t}{3}\right)\frac{2\pi}{\sqrt{3}\Omega_t^2}\right)\\
&=\sum\limits_{c\in V}\left(\frac{t}{\mf{b}}\right)_3\left(\frac{1}{2}\frac{\epsilon  \wp'(\frac{\Omega_t}{3}, \mathcal{L}_t)}{\wp(\frac{\Omega_t}{3}, \mathcal{L}_t)-\wp(\frac{c\Omega_t}{M}, \mathcal{L}_t)}+ \epsilon \zeta\left(\frac{\Omega_t}{3}, \mathcal{L}_t\right)-\epsilon \frac{2\pi}{3\sqrt{3}\Omega_t}\right).
\end{align*}
\end{SMALL}
Here, we used the key property that $c\in V$ then $-c\in V$, $\left(\frac{2}{\mf{b}}\right)_3=\left(\frac{c}{2}\right)_3=\left(\frac{-c}{2}\right)_3$ and that $\zeta$ and $\wp'$ are odd functions, whereas $\wp$ is an even function. Recall that $\Omega_1$ and $\mathcal{L}_1$ correspond to the equation $y^2=4x^3-3^3$. Thus by computation, we get $\wp(\frac{\Omega_1}{3}, \mathcal{L}_1)=3$ and $\wp'(\frac{\Omega_1}{3}, \mathcal{L}_1)=9$. Now, we can write $\Omega_t=\frac{\Omega_1}{\sqrt[3]{t}}$, so that  $\wp(\frac{\Omega_t}{3}, \mathcal{L}_t)=(\sqrt[3]{t})^{2} \wp(\frac{\Omega_1}{3}, \mathcal{L}_1)=3\sqrt[3]{t^2}$. Furthermore, we have $(x,y)=(\wp(\frac{\Omega_t}{3}, \mathcal{L}_t), \wp'(\frac{\Omega_t}{3}, \mathcal{L}_t))$ satisfying the equation for $y^2=4x^3-t^23^3$. It follows that for $t=2$, we have $\wp'(\frac{\Omega_2}{3}, \mathcal{L}_2)=18$. Similarly, for $t=4$, we have $\wp'(\frac{\Omega_4}{3}, \mathcal{L}_4)=36$.

 By applying formulae (3.2) and (3.3) of \cite[p. 126]{stephens}, we obtain
\begin{equation}\label{eq2.2.4}\zeta(z+1,\o)=\zeta(z,\o)+\frac{2\pi}{\sqrt{3}}, \;\;\;\zeta(z+\omega,\o)=\zeta(z,\o)+\frac{2\pi}{\sqrt{3}}\overline{\omega}.
\end{equation}
Letting $z=-\frac{1}{3}$ in \eqref{eq2.2.4} gives
$$\zeta\left(\frac{2}{3}, \o\right)+\zeta\left(\frac{1}{3}, \o\right)=\frac{2\pi}{\sqrt{3}}.$$
But we have $\zeta\left(\Omega_t z, \mathcal{L}_t\right)=\frac{1}{\Omega_t}\zeta\left(z, \o\right)$, so
\begin{equation}\label{eq1}\zeta\left(\frac{2\O}{3}, \mathcal{L}_t\right)+\zeta\left(\frac{\Omega_t}{3}, \mathcal{L}_t\right)=\frac{2\pi}{\sqrt{3}\Omega_t}.
\end{equation}
On the other hand, we have
$$\zeta(2z, \mathcal{L}_t)=2\zeta(z,\mathcal{L}_t)+\frac{\wp''(z, \mathcal{L}_t)}{2\wp'(z,\mathcal{L}_t)},$$
and by differentiating the equation $\left(\wp'(z,\mathcal{L}_t)\right)^2=4\wp^3(z,\mathcal{L}_t)-t^2 3^3$, we get $\wp''(z,\mathcal{L}_t)=6\wp^2(z,\mathcal{L}_t)$. Thus plugging in the values of $\wp\left(\frac{\Omega_t}{3},\mathcal{L}_t\right)$ and $\wp'\left(\frac{\Omega_t}{3},\mathcal{L}_t\right)$, we obtain
\begin{equation}\label{eq2}\zeta\left(\frac{2\O}{3}, \mathcal{L}_t\right)-2\zeta\left(\frac{\Omega_t}{3}, \mathcal{L}_t\right)=\frac{\wp''\left(\frac{\Omega_t}{3}, \mathcal{L}_t\right)}{2\wp'\left(\frac{\Omega_t}{3},\mathcal{L}_t\right)}=\frac{6\wp^2\left(\frac{\Omega_t}{3}, \mathcal{L}_t\right)}{2\wp'\left(\frac{\Omega_t}{3},\mathcal{L}_t\right)}=3\sqrt[3]{t}.
\end{equation}
Now, solving \eqref{eq1} and \eqref{eq2} gives
$$\zeta\left(\frac{\Omega_t}{3}, \mathcal{L}_t\right)=\frac{2\pi}{3\sqrt{3}\Omega_t}-\sqrt[3]{t}.$$
Hence,
\begin{align}\label{sumE}\nonumber \sum\limits_{c\in V}\left(\frac{t}{\mf{b}}\right)_3\mathcal{E}_1^*\left(\frac{c\Omega_t}{M}+\frac{\Omega_t}{3}, \mathcal{L}_t\right)&=\sum\limits_{c\in V}\left(\frac{t}{\mf{b}}\right)_3\left(\frac{1}{2}\frac{\epsilon \wp'\left(\frac{\Omega_t}{3},\mathcal{L}_t\right)}{\wp\left(\frac{\Omega_t}{3},\mathcal{L}_t\right)-\wp\left(\frac{c\Omega_t}{M},\mathcal{L}_t\right)}-\epsilon \sqrt[3]{t}\right)\\
&=\sqrt[3]{t}\left(\sum\limits_{c\in V}\frac{1}{2}\left(\frac{t}{\mf{b}}\right)_3\left(\frac{9\epsilon }{3-\wp\left(\frac{c\Omega_1}{M},\mathcal{L}_1\right)}\right)-\epsilon \left(\frac{t}{\mf{b}}\right)_3\right)\\
&=\left\{ \begin{array}{ll}
\frac{1}{2}\sum\limits_{c\in V}\frac{3\epsilon}{1-\wp\left(\frac{c\Omega}{M},\mathcal{L}\right)}-\epsilon \#(V)&\mbox{if $t=1$} \\
\nonumber \sqrt[3]{t}\left(\frac{1}{2}\sum\limits_{c\in V}\left(\frac{t}{\mf{b}}\right)_3\frac{3\epsilon}{1-\wp\left(\frac{c\Omega}{M},\mathcal{L}\right)}\right) &\mbox{otherwise,}
       \end{array} \right.
\end{align}
where $\Omega$ and $\mathcal{L}$ correspond to the equation $y^2=4x^3-1$, so that $\Omega_1=\frac{\Omega}{\sqrt{3}}$. Here, we also used the fact that $\sum_{c\in V}\left(\frac{t}{\mf{b}}\right)_3=0$ if $2\mid t$, since
$$\#\left(\Gal(K(\mf{f})/K(\sqrt[3]{2},\sqrt[3]{p_1}, \ldots ,\sqrt[3]{p_n}))\right)=\frac{\#(V)}{3},$$
and thus there exists $c\in V$ with $\left(\frac{t}{\mf{b}}\right)_3=\left(\frac{c}{t}\right)_3\neq 1$.
But we know from \cite[Lemma 2.5]{kl} or \cite[Lemma 3]{szhang} that 
$$\ord_3\left(1-\wp\left(\frac{c\Omega}{M},\mathcal{L}\right)\right)=1/3,$$
since $(M,3)=1$. The claim now follows on noting that, by Lemma \ref{lemV}, we always have $\ord_3(\#(V))\geq 1$ except when $t=1$ and $n=r$.
\end{proof}

Given any character $\chi: (\Z/3\Z)^n\to \C^\times$, we define
\begin{equation}\label{PhiChi}\Phi_{tD}^{(\chi)}=\sum_{\alpha\in (\Z/3\Z)^n}\chi(\alpha)\frac{L_S(\overline{\psi}_{tD_\alpha},1)}{\Omega_t}.
\end{equation}

\begin{cor}\label{Chi} For any character $\chi: (\Z/3\Z)^n\to \C^\times$, we have
\begin{enumerate}
\item $\ord_3\left(\Phi^{(\chi)}_{tD}\right)\geq n-1$ if $t=1$ and $n=r$.
\item $\ord_3\left(\Phi^{(\chi)}_{tD}\right)\geq n-\frac{1}{3}$ otherwise.

\end{enumerate}
\end{cor}

\begin{proof}For any $i\in\{1,\ldots , n\}$, let $\delta_i=(0,\ldots ,0, 1,0,\ldots ,0)\in (\Z/3\Z)^n$ whose $i$th coordinate is equal to $1$ . Note that 
$\varphi^{(\chi,\mf{b})}: \alpha\mto \chi(\alpha)\left(\frac{D_\alpha}{\mf{b}}\right)_3$ gives a $1$-dimentional character of $(\Z/3\Z)^n$. Thus considering the inner product of $\varphi^{(\chi,\mf{b})}$ and the trivial character gives 
$$\sum_{\alpha \in (\Z/3\Z)^n}\varphi^{(\chi,\mf{b})}(\alpha)=\left\{ \begin{array}{ll}
 3^n &\mbox{if $c\in V^{(\chi)}$} \\
 0 &\mbox{otherwise,}
       \end{array} \right.$$
where $V^{(\chi)}=\{c\in \mc{C}: \left(\frac{p_i}{\mf{b}}\right)=\chi(\delta_i)^2 \text{ for all }i\}$. Suppose there is an element $c'\in \mc{C}$ such that $\left(\frac{p_i}{c'}\right)_3=\chi(\delta_i)^2$ for all $i$. Then $V^{(\chi)}=c'V$, so that in particular $\#(V^{(\chi)})=\#(V)$. If no such $c'$ exists, then $\#(V^{(\chi)})=0$. In either case, we have $\ord_3(\#(V^{(\chi)}))\geq \ord_3(\#(V))$, and the rest of the proof is a straightforward adaptation of the proof of Theorem \ref{thmPhi}. 
\end{proof}

We are now ready to prove our first main result. Recall that given an integer $m$, $k(m)$ denotes the number of distinct prime factors dividing $m$ in $\Z$. 
\begin{thm}\label{main1}Let $D> 1$ be any integer with $k(D)=n$ and $(D,6)=1$. Let $\e_{tD}=1$ if $t=1$ and $D$ has a prime factor which is inert in $K$, and let $\e_{tD}=0$ otherwise. Then
 $$\ord_3\left(L^{(\mathrm{alg})}(C_{tD}, 1)\right)\geq n-\e_{tD}.$$
\end{thm}

\begin{proof} We will prove this by induction on $n$. Suppose first that $n=1$, so that $D=p$ or $p^2$. We have $\frac{L(\overline{\psi}_{1},1)}{\Omega_1}=L^{(\mathrm{alg})}(C_1,1)=\frac{1}{3}$, $\frac{L(\ovl{\psi}_{2},1)}{\Omega_2}=L^{(\mathrm{alg})}(C_2,1)=\frac{1}{2}$ and $\frac{L(\ovl{\psi}_{4},1)}{\Omega_4}=L^{(\mathrm{alg})}(C_4,1)=1$. In the case when $p$ is inert, we have
\begin{align*}\Phi_{tp}&=\frac{L_S(\ovl{\psi}_{t},1)}{\Omega_t}+\frac{L_S(\ovl{\psi}_{t p},1)}{\Omega_t}+\frac{L_S(\ovl{\psi}_{tp^2},1)}{\Omega_t}\\
&=\left(1-\frac{\ovl{\psi}_t((p))}{\mathrm{N}p}\right)\frac{L(\ovl{\psi}_{t},1)}{\Omega_t}+\frac{L(\ovl{\psi}_{tp},1)}{\Omega_t}+\frac{L(\ovl{\psi}_{t p^2},1)}{\Omega_t}\\
&=\left(\frac{p^2+p}{p^2}\right)\frac{L(\ovl{\psi}_{t},1)}{\Omega_t}+\frac{L(\ovl{\psi}_{t p},1)}{\Omega_t}+\frac{L(\ovl{\psi}_{tp^2},1)}{\Omega_t},
\end{align*}
since we have $\left(\frac{a}{b}\right)_3=1$ for any coprime integers $a$ and $b$ with $(b,3)=1$ (see, for example, \cite[Proposition 7.7]{Lemm}).
If $p$ is a split prime, say $p\o=(\pi)(\ovl{\pi})$ with $\pi\equiv 1\bmod 3\o$, then we have
\begin{align*}\Phi_{tp}=\left(\frac{\pi-\left(\frac{t}{\pi}\right)_3}{\pi}\right)\left(\frac{\ovl{\pi}-\left(\frac{t}{\ovl{\pi}}\right)_3}{\ovl{\pi}}\right)\frac{L(\ovl{\psi}_{t},1)}{\Omega_t}+\frac{L(\ovl{\psi}_{t p},1)}{\Omega_t}+\frac{L(\ovl{\psi}_{tp^2},1)}{\Omega_t}.
\end{align*}
Thus $\ord_3\left(\left(\frac{\pi-\left(\frac{t}{\pi}\right)_3}{\pi}\right)\left(\frac{\ovl{\pi}-\left(\frac{t}{\ovl{\pi}}\right)_3}{\ovl{\pi}}\right)\right)\geq \frac{1}{2}+\frac{1}{2}=1$. If $t=1$, then $\left(\frac{t}{\pi}\right)_3=1$, so $\ord_3\left(\left(\frac{\pi-1}{\pi}\right)\left(\frac{\ovl{\pi}-1}{\ovl{\pi}}\right)\right)\geq 1+1=2$. Let $\chi$ be a character on $\Z/3\Z$ sending $1$ to $\omega$. Then 
$$\Phi^{(\chi)}_{tp}-\omega^2 \Phi_{tp}=(1-\omega)\frac{L_S(\ovl{\psi}_{t},1)}{\Omega_t}+(\omega-\omega^2)\frac{L(\ovl{\psi}_{tp},1)}{\Omega_t}.$$

Then we have $\ord_3\left(\Phi^{(\chi)}_{tp}-\omega^2 \Phi_{tp}\right)\geq \frac{2}{3}(1-\varepsilon_{tp})$ and $\ord_3(1-\omega)=\frac{1}{2}$, so $\ord_3\left(\frac{L(\ovl{\psi}_{tp},1)}{\Omega_t}\right)\geq \frac{1}{6}-\frac{2}{3}\varepsilon_{tp}$. Since $L^{(\mathrm{alg})}(C_{tp},1) \in \Q$ by \cite{bsd2} (in fact an integer by \cite{stephens} in our case), its $3$-adic valuation must be an integer. Thus combining with Theorem \ref{thmPhi}, we obtain $\ord_3\left(L^{(\mathrm{alg})}(C_{tp},1)\right)\geq 1-\varepsilon_{tp}$, as required. The same holds for $L^{(\mathrm{alg})}(C_{tp^2},1)$.

Now assume $n>1$. Given $\alpha,\beta\in (\Z/3\Z)^n$ we write $\beta<\alpha$ if $k(D_\beta)<k(D_\alpha)$.
$$\Phi_{tD}=\frac{L_S(\ovl{\psi}_{t},1)}{\Omega_t}+\sum_{\beta<\alpha}\frac{L_S(\ovl{\psi}_{tD_\beta},1)}{\Omega_t}+\sum_{\alpha}\frac{L(\ovl{\psi}_{tD_\alpha},1)}{\Omega_t},$$
where the last summands are primitive. We know that  
\begin{align*}\frac{L_S(\ovl{\psi}_{t},1)}{\Omega_t}&=\prod_{\substack{p\mid D \\ \text{inert}}}\left(\frac{p^2+p}{p}\right)\prod_{\substack{p\mid D\\ p\o=(\pi)(\ovl{\pi})}}\left(\frac{\pi-\left(\frac{t}{\pi}\right)_3}{\pi}\right)\left(\frac{\ovl{\pi}-\left(\frac{t}{\ovl{\pi}}\right)_3}{\ovl{\pi}}\right)\frac{L(\ovl{\psi}_{t},1)}{\Omega_t}
\end{align*}
and thus $\ord_3\left(\frac{L_S(\ovl{\psi}_{t},1)}{\Omega_t}\right)\geq n-\e_{tD}$ (note again that $\left(\frac{t}{\pi}\right)_3=1$ when $t=1$). Next, given  $\beta<\alpha$, we have 
\begin{small}
\begin{align*}\frac{L_S(\ovl{\psi}_{tD_\beta},1)}{\Omega_t}&=\prod_{\substack{p\mid (D_\alpha/D_\beta)\\\text{inert}}}\left(\frac{p^2+p}{p}\right) \prod_{\substack{p\mid (D_\alpha/D_\beta)\\  p\o=(\pi)(\ovl{\pi}) }}\left(\frac{\pi-\left(\frac{tD_\beta}{\pi}\right)_3}{\pi}\right)\left(\frac{\ovl{\pi}-\left(\frac{tD_\beta}{\ovl{\pi}}\right)_3}{\ovl{\pi}}\right)\frac{L(\ovl{\psi}_{tD_\beta},1)}{\Omega_t}.
\end{align*}
\end{small}

By the induction hypothesis, we have $\ord_3\left(\frac{L(\ovl{\psi}_{tD_\beta},1)}{\Omega_t}\right)=r_\beta+s_\beta +\e_{tD_\beta}-1$, where $r_\beta$ and $s_\beta$ denote the numbers of inert and split primes dividing $D_\beta$, respectively. It follows that 
\begin{align*}\ord_3\left(\frac{L_S(\ovl{\psi}_{tD_\beta},1)}{\Omega_t}\right)&=(r-r_\beta)+\frac{1}{2}(2s-2s_\beta)+r_\beta+s_\beta-\e_{tD_\beta}\\
&=r+s-\e_{tD_\beta}\geq n-\e_{tD}.
\end{align*}
Thus we have shown that $\ord_3\left(\sum_{\alpha}\frac{L(\ovl{\psi}_{tD_\alpha},1)}{\Omega_t}\right)\geq n-\e_{tD}.$
We claim that this lower bound applies to the individual summand, which proves the theorem. We will prove this claim by applying another layer of induction.

Suppose $D=p_1^{e_1}\cdots p_n^{e_n}$. Given a subset $I\subset \{1,\ldots ,n\}$, define
$$\Sigma_I:=\sum\limits_{\substack{\alpha_i=e_i \\i\in I}}\frac{L(\overline{\psi}_{tD_\alpha},1)}{\Omega_t}.$$
We claim that $\ord_3\left(\sum\limits_{\substack{\alpha_i=e_i \\i\in I}}\frac{L(\overline{\psi}_{tD_\alpha},1)}{\Omega_t}\right)\geq  \ord_3(\Phi_{tD})-\frac{1}{2}$ for any $I$. We will prove this by induction on $\#(I)$. First, consider the case when $\#(I)=1$. Without loss of generality, we may assume $i=1$ and first suppose $e_1=1$. Let $\chi$ be a character  of $(\Z/3\Z)^n$ sending $\delta_1$ to $\omega$ an $\delta_i$ to $1$ for $i\neq 1$. Then we see that
$$\ord_3\left((\omega-\omega^2)\sum_{\alpha_1=1}\frac{L(\ovl{\psi}_{tD_\alpha},1)}{\Omega_t}\right)\geq \ord_3\left(\Phi^{(\chi)}_{tD}-\omega^2 \Phi_{tD}\right).$$
Thus Corollary \ref{Chi} gives $\mathrm{ord}_3\left(\sum_{\alpha_1=1}\frac{L(\ovl{\psi}_{tD_\alpha},1)}{\Omega_t}\right)\geq \ord_3(\Phi_{tD})-\frac{1}{2}$. Similarly, if $e_1=2$, we can show $\mathrm{ord}_3\left(\sum_{\alpha_1=2}\frac{L(\ovl{\psi}_{tD_\alpha},1)}{\Omega_t}\right)\geq \ord_3(\Phi_{tD})-\frac{1}{2}$.

Suppose now that the claim holds for any $I\subset \{1,\ldots , n\}$ with $1\leq \#(I)\leq k$. Then we consider the case when $\#(I)=k+1$, and without loss of generality, we may assume $I=\{1,\ldots , k+1\}$. By the induction hypothesis, 
$\mathrm{ord}_3(\Sigma_{\{1,\ldots , k\}})\geq  \ord_3(\Phi_{tD})-\frac{1}{2}$ and $\mathrm{ord}_3(\Sigma_{\{2,\ldots , k+1\}})\geq  \ord_3(\Phi_{tD})-\frac{1}{2}$. Now, 
\begin{align*}\Sigma_{\{1,\ldots , k\}}-\Sigma_{\{2,\ldots k+1\}}&=\sum\limits_{\substack{\alpha_i=e_i i\in \{2,\ldots k\}\\ \alpha_1=e_1, \alpha_{k+1}\neq e_{k+1}}}\frac{L(\overline{\psi}_{tD_\alpha},1)}{\Omega_t}-\sum\limits_{\substack{\alpha_i=e_i i\in \{2,\ldots k\} \\ \alpha_1\neq e_1, \alpha_{k+1}= e_{k+1}}}\frac{L(\overline{\psi}_{tD_\alpha},1)}{\Omega_t}\\
&=A-B,
\end{align*}
say. Now, $A+B+\Sigma_I=\Sigma_{\{2,\ldots , k\}}$ so $\mathrm{ord}_3(A+B+\Sigma_I)\geq  \ord_3(\Phi_{tD})-\frac{1}{2}$. On the other hand,
\mbox{$\Sigma_{\{1,\ldots k\}}+\Sigma_{\{2,\ldots k+1\}}=A+B+2\Sigma_I$} so $\mathrm{ord}_3(A+B+2\Sigma_I)\geq  \ord_3(\Phi_{tD})-\frac{1}{2}$.  It follows that $\mathrm{ord}_3(\Sigma_I)\geq  \ord_3(\Phi_{tD})-\frac{1}{2}$, as required.

Hence applying the claim to $I=\{1, \ldots, n\}$, we see that for any $\alpha\in (\mathbb{Z}/3\mathbb{Z})^n$ with $k(D_\alpha)=n$, we have
 $$\mathrm{ord}_3\left(\frac{L(\overline{\psi}_{tD_\alpha},1)}{\Omega_t}\right)\geq  \ord_3(\Phi_{tD})-\frac{1}{2}.$$
Since we know $\ord_3\left(L^{(\mathrm{alg})}(C_{tD},1)\right)$ is an integer, it follows from Theorem \ref{thmPhi} that $\ord_3(L^{\mathrm{alg}}(C_{tD},1)))\geq n-\e_{tD}$, as required.
\end{proof}

\begin{remark} If $k(D)=1$, $t\in \{2,4\}$ and $D\equiv 2,4,5,7\bmod 9$, the proof can be shortened significantly, since by a root number consideration (see Corollary \ref{root}), one of the two primitive Hecke $L$-values vanishes in the sum $\Phi_{tp}$. In particular, there is no need to introduce \eqref{PhiChi}.
\end{remark}

\section{The $3$-descent and the integrality of $\mathcal{S}_N$}\label{section2a}
In this section, we give a $3$-descent argument, and then discuss a consequence of Theorem \ref{m1}. Given any cube-free positive integer $N$ (not necessarily prime to $3$), we write 
$$E_N: y^2z=x^3-2^43^3N'^2z^3,$$
where $N'=\frac{N}{3^3}$ if $3^2\mid\mid N$ and $N'=N$ otherwise.
This is birationally equivalent to $C_N: x^3+y^3=N$, and over $\Q$ there is a $3$-isogeny $\phi$ to 
$$E_N': y^2z=x^3+2^4N^2z^3.$$

\begin{thm}\label{3descent} Let $C_{N}: x^3+y^3=N$, where $N> 1$ is a cube-free integer. Then the $3$-part of the Tate--Shafarevich group $\Sha(C_{N})[3]$ of $C_{N}$ over $\Q$ satisfies
$$\ord_3\left(\#\left(\Sha(C_{N})[3]\right) \right)\geq r(N)-t(N)-1-\rank(C_N),$$
where $r(N)$ is the number of distinct primes dividing $N$ which are inert in $\Q(\sqrt{-3})$, $t(N)=1$ if $N\equiv \pm 1 \bmod 9$, $t(N)=-1$ if $\ord_3(N)=1$ and $t(N)=0$ otherwise. 
\end{thm}

\begin{proof}

The isogeny $\phi: E_N \to E_N'$ is explicitly given in \cite[D\'{e}finition 1.1]{satge}. The result of Satg\'{e} \cite[Proposition 2.8]{satge} based on the work of Cassels \cite{cassels} gives us
$$
\dim_{\mb{F}_3}(\Sel(\Q, E_N[\phi]))-\dim_{\mb{F}_3}(\Sel(\Q, E_N'[\hat{\phi}]))=r(N)-t(N),
$$
where $\hat{\phi}$ denotes the dual isogeny, and we regard the Selmer groups of $\phi$ and $\hat{\phi}$ as vector spaces over the finite field $\mb{F}_3$ with $3$ elements. Recall that we have an exact sequence
$$0\to E_N'(\Q)/\phi(E_N(\Q)) \to \Sel(\Q, E_N[\phi])\to \Sha(E_N)[\phi]\to 0.$$
Now, we know that $E_N(\Q)_{\mathrm{tor}}=0$, $E_N'(\Q)_{\mathrm{tor}}=\{(0,1,0),(0,\pm 2^2 N,1)\}\simeq \Z/3\Z$. Recall also that $\rank (E_{N})=\rank (E'_{N})$ since they are isogeneous. The result now follows from the exact sequence
$$0\to \Sha(E_N)[\phi]\to \Sha(E_N)[3]\xto{\phi} \Sha(E_N')[\hat{\phi}],$$
which shows that $\ord_3\left(\#\left(\Sha(E_N)[3]\right)\right)\geq \ord_3\left(\#\left(\Sha(E_N)[\phi]\right)\right)$.
\end{proof}

Next, we discuss a corollary of Theorem \ref{m1}. Recall that for $N>2$,
$$\mathcal{S}_N=\frac{L^{(\mathrm{alg})}(C_N,1)}{\prod_{q \mathrm{ bad}}c_q}$$ 
is defined so that if $L(C_N,1)\neq 0$, the Birch--Swinnerton-Dyer conjecture predicts 
$$\mathcal{S}_N=\#(\Sha(C_N)).$$
The Tamagawa number divisibility in Theorem \ref{m1}, when combined with \cite[Theorem 1.1--Theorem 1.3]{rosu}, immediately gives the following.

\begin{cor}Let $N>2$ be a cube-free positive integer with $(N,3)=1$. Then 
\begin{enumerate}
\item $\mathcal{S}_N$ is a perfect square if $N$ is a product of split primes.
\item $2\mathcal{S}_N\in \Z$ if $N\equiv 2,7\bmod 9$.
\item $\mathcal{S}_N\in \Z$ otherwise.
\end{enumerate} 
Furthermore, we have the trace formula
$$\mathcal{S}_N=\frac{1}{3\prod_{q \mathrm{ bad}}c_q}\mathrm{Tr}_{K(j(\o_{3N}))/K}\left(\sqrt[3]{N}\frac{\Theta_K(N\omega)}{\Theta_K(\omega)}\right)$$
where $\Theta_K(z)=\sum_{a,b\in \Z}e^{2\pi i z (a^2+b^2-ab)}$ is the theta function of weight one associated to $K$, and $K(j(\o_{3N}))$ is the ring class field of the order $\o_{3N}=\Z+3N\o$.
\end{cor}

\begin{proof} If $N$ is a product of split primes, then we have $N\equiv 1\bmod 3$, and $\mathcal{S}_N\neq 0$ only if $N\equiv 1\bmod 9$ by Proposition \ref{lroot}. It follows from Theorem \ref{m1} and Lemma \ref{tam} that $3\mathcal{S}_N\in \Z$. By \cite[Theorem 1.2]{rosu}, we have $\mathcal{S}_N=\frac{k^2}{3^{2\ell}}$ for some $k, \ell\in \Z$. If the square of a rational number $\left(\frac{k}{3^{\ell}}\right)^2$ is an element of $\frac{1}{3}\Z$, it must be an element of $\Z$, and thus (1) follows. Part (2) and (3) follow from Theorem \ref{m1}, Lemma \ref{tam} and the fact that $L^{(\mathrm{alg})}(C_N,1)\in \Z$ shown in \cite{stephens} or in Section \ref{section3} of this paper. Note that if $r(N)=t(N)=1$, we have $\mathcal{S}_N=0$ since the global root number $\epsilon(C_N/\Q)=-1$ by Proposition \ref{lroot}.  Finally, the trace formula is given in \cite[Theorem 1.3]{rosu}. 
\end{proof}

For a more detailed description of $k$ and $\ell\in \Z$ satisfying $\mathcal{S}_N=\frac{k^2}{3^{2\ell}}$ in terms of theta functions of weight $1/2$, we refer to \cite[Theorem 1.2]{rosu}.

\section{The $2$-adic valuation of the algebraic part of central $L$-values}\label{section3}

 Let $N>2$ be a cube-free integer with $(N,3)=1$, and let $C_{N}: x^3+y^3=N$. The aim of this section is to give another proof of the integrality result  
 $$L^{(\mathrm{alg})}(C_{N}, 1)\in \Z$$
of Stephens \cite{stephens} by directly studying the remaining $2$-adic valuation of $L^{(\mathrm{alg})}(C_{N}, 1)$. The approach given here has the advantage that it gives us an insight into how the $L$-values can be divisible by a large power of $2$ for quadratic or sextic twists of the curve $y^2=x^3-2^4 3^3$. Recall that $\psi_{N}$ denotes the Grossencharacter of $C_N $ over $K$. We write $N=tD$, where $t\in \{1,2,4\}$, $2\nmid D$ and $k(D)=n$. Let $E_{tD}$ be given by \eqref{eqE} as before. We write $\Omega_{tD}$ for the real period introduced in \eqref{periods}, and let $\mathcal{L}_{tD}=\Omega_{tD}\o$ denote its period lattice. Recall that the conductor of $\psi_{tD}$ divides $\mathfrak{f}=f\o$, where $f = \mathrm{rad}(3tD)$. Let $S$ be the set of primes of $K$ dividing $3tD$. 

Then again by \cite[Proposition 48]{coates1}, we have $\Gal(K(C_{tD, \mf{f}})/K)=\Gal(K(\mf{f})/K)\simeq (\o/ M \o)^\times$, where $M=\mathrm{rad}(tD)$. We set $\mathcal{C}$ to be a set of elements of $\o$ prime to $M$ such that $c\mod M$ runs over $(\o/M\o)^\times$ precisely once as $c$ runs over $\mathcal{C}$, and such that $c\in \mathcal{C}$  implies $-c\in \mathcal{C}$. Define
\begin{align*}
 \mathcal{B}&=\{(3c+\epsilon M) : c\in \mathcal{C}\} \mbox{ if $t=1$, and}\\
\mathcal{B}&= \{(6c+\epsilon M) : c\in \mathcal{C}\} \mbox{ if $2\mid t$,} 
\end{align*}
where the sign $\epsilon\in \{\pm 1\}$ of $M$ is again chosen so that $\epsilon M\equiv 1\bmod 3$.
Suppose first that $2\mid t$. Then we have $(\o/ M \o)^\times\simeq (\o/2\o)^\times \times (\o/D\o)^\times$, and we take $\sigma_\mf{t}\in \Gal(K(\mf{f})/K)$ to be an element corresponding to a generator of $(\o/2\o)^\times$. Then the Artin symbol $\sigma_\mathfrak{b}\sigma_{\mf{t}^i}$ in the extension $K(\mf{f})/K$ runs over the Galois group $\Gal(K(\mathfrak{f})/K)$ precisely three times as $\mathfrak{b}$ runs over $\mathcal{B}$ and $i$ runs over $\{0,1,2\}$.
Thus, again by \cite[Proposition 5.5]{gol-sch}, for any $\alpha\in (\Z/3 \Z)^n$ we have
\begin{align*}
\frac{3 L_S(\overline{\psi}_{tD_\alpha},1)}{\Omega_{tD_\alpha}}&=\frac{1}{6D}\sum_{i=0}^2\sum_{\mathfrak{b}\in \mathcal{B}}\mathcal{E}^*_1\left(\frac{\Omega_{tD_\alpha}}{6D}, \mathcal{L}_{tD_\alpha}\right)^{\sigma_{\mathfrak{b}}\sigma_{\mf{t}^i}}.
\end{align*}
Similarly, if $t=1$, we have
\begin{align*}
\frac{L_S(\overline{\psi}_{D_\alpha},1)}{\Omega_{D_\alpha}}&=\frac{1}{3D}\sum_{\mathfrak{b}\in \mathcal{B}}\mathcal{E}^*_1\left(\frac{\Omega_{D_\alpha}}{3D}, \mathcal{L}_{D_\alpha}\right)^{\sigma_\mathfrak{b}}.
\end{align*}
It follows that in both cases we have
$$\ord_2\left(\frac{L_S(\overline{\psi}_{tD_\alpha},1)}{\Omega_{tD_\alpha}}\right)\geq \ord_2\left(\frac{1}{f}\sum_{\mathfrak{b}\in \mathcal{B}}\mathcal{E}^*_1\left(\frac{\Omega_{tD_\alpha}}{f}, \mathcal{L}_{tD_\alpha}\right)^{\sigma_\mathfrak{b}}\right).$$
Thus, if we can show $\ord_2\left(\frac{1}{f}\sum_{\mathfrak{b}\in \mathcal{B}}\mathcal{E}^*_1\left(\frac{\Omega_{tD_\alpha}}{f}, \mathcal{L}_{tD_\alpha}\right)^{\sigma_\mathfrak{b}}\right)\geq 0$, then it follows that $\ord_2\left(L^{(\mathrm{alg})}(C_{N}, 1)\right)\geq 0$.
We define
$$V=\{c\in \mathcal{C}: \left(\frac{p}{\mathfrak{b}}\right)_3=1 \text{ for all }p\mid D \text{, where }\mathfrak{b}\in \mathcal{B} \text{ corresponds to } c\}.$$

Recall that
$$\Phi_{tD}=\sum_{\alpha\in (\Z/3\Z)^n}\frac{L_S(\overline{\psi}_{tD_\alpha},1)}{\Omega_t}.$$

\begin{thm}\label{thmPhi2} For any cube-free odd integer $D>1$ prime to $3$, we have
$$\ord_2\left(\Phi_{D}\right)\geq 0, \quad \quad\ord_2\left(\Phi_{2D}\right)\geq -\frac{2}{3} \quad \quad \text{and} \quad \quad \ord_2\left(\Phi_{4D}\right)\geq -\frac{1}{3}.$$
\end{thm}

\begin{proof}
Note first that $\ord_2(3\Phi_{tD})=\ord_2(\Phi_{tD})$. Following the methods of the proof of Theorem \ref{thmPhi}, we obtain
\begin{align*}
\ord_2\left(\Phi_{tD}\right)&\geq\ord_2\left(\frac{1}{f}\sum_{\mathfrak{b}\in \mathcal{B}}\left(\sum_{\alpha\in (\Z/3\Z)^n} \left(\frac{D_\alpha}{\mathfrak{b}}\right)_3 \right)\mathcal{E}^*_1\left(\frac{\Omega_t}{f},\mathcal{L}_t\right)^{\sigma_\mathfrak{b}}\right)\\
&=\ord_2\left(\frac{3^n}{f}\sum_{c\in V}\mathcal{E}^*_1\left(\frac{\Omega_t}{f},\mathcal{L}_t\right)^{\sigma_\mathfrak{b}}\right),
\end{align*}
because $\sum_{\alpha\in (\Z/3\Z)^n} \left(\frac{D_\alpha}{\mathfrak{b}}\right)_3=3^n$ if $\left(\frac{p}{\mathfrak{b}}\right)_3=1$ for all $p\mid D$ and $\sum_{\alpha\in (\Z/3\Z)^n} \left(\frac{D_\alpha}{\mathfrak{b}}\right)_3=0$ otherwise. Since $f=3D$ or $6D$ according as $t=1$ or $t\neq 1$, it suffices to show

$$\ord_2\left(\sum_{c\in V}\left(\frac{t}{\mf{b}}\right)_3\mathcal{E}^*_1\left(\frac{\epsilon \Omega_t}{3}+\frac{c\Omega_t}{D},\mathcal{L}_t\right)\right)\geq
 \left\{ \begin{array}{ll}
0 &\mbox{if $t=1$} \\
1/3 &\mbox{if $t=2$} \\
2/3 & \mbox{if $t=4$.}
       \end{array} \right.
$$

Recall that $\Omega_t=\frac{\Omega_1}{\sqrt[3]{t}}$, and that $\mathcal{E}^*_1$ is homogeneous of degree $-1$. Since we still have $-c\in V$ whenever $c\in V$, \eqref{sumE} gives

$$
\sum_{c\in V}\left(\frac{t}{\mf{b}}\right)_3\mathcal{E}^*_1\left(\frac{\epsilon \Omega_t}{3}+\frac{c\Omega_t}{D},\mathcal{L}_t\right)=
 \left\{ \begin{array}{ll}
\frac{1}{2}\sum\limits_{c\in V}\frac{9\epsilon }{3-\wp\left(\frac{c\Omega_1}{D},\mathcal{L}\right)}-\epsilon \#(V)&\mbox{if $t=1$} \\
\frac{\sqrt[3]{t}}{2}\sum\limits_{c\in V}\left(\frac{t}{\mf{b}}\right)_3\left(\frac{9\epsilon}{3-\wp\left(\frac{c\Omega_1}{D},\mathcal{L}_1\right)}\right)& \mbox{otherwise.}
       \end{array} \right.
$$

Recall that $\Omega_1$ and $\mathcal{L}_1$ correspond to the equation $y^2=4x^3-3^3$ with a Weierstrass form $A: Y^2+Y=X^3-7$ with discriminant $3^9$, so that in particular it is minimal at the prime $2$. We know $\wp(\frac{\Omega_1}{3}, \mathcal{L}_1)=3$ and $\wp'(\frac{\Omega_1}{3}, \mathcal{L}_1)=9$, and thus $3-\wp\left(\frac{c\Omega_1}{D}, \mathcal{L}_1\right)=X(Q)-X(P)$, where $Q=(3,4)$ is a point on $A$ of order $3$ and $P$ is a point of order $D$. We claim that $\ord_2\left(X(Q)-X(P))\right)=0$. Indeed, since $A$ has good reduction at $2$ and $P$ is a torsion point of order prime to $2$, we have $\ord_2(X(P))\geq 0$.  Suppose for a contradiction that $\ord_2(X(Q)-X(P))>0$. Then we have $X(\widetilde{Q})=X(\widetilde{D})$, where $\widetilde{\phantom{Q}}$ denotes reduction modulo $2$. It follows that either $P-Q$ or $P+Q$ lies in the reduction modulo $2$ map, and thus it must correspond to an element in the formal group of $A$ at $2$. But neither can be a torsion of order a power of $2$, since $3$ and $D$ are both prime to $2$. Finally, $\wp$ is an even function, $-1$ is a cubic residue, and $c\in V$ implies $-c\in V$. Therefore, we have
$$\ord_2\left(\sum\limits_{c\in V}\left(\frac{t}{\mf{b}}\right)_3\frac{9\epsilon}{3-\wp\left(\frac{c\Omega_1}{D},\mathcal{L}_1\right)}\right)\geq 1.$$
The result now follows on noting that $\ord_2\left(\frac{\sqrt[3]{t}}{2}\right)=\ord_2(t)/3-1$.
\end{proof}

Recall that we defined
$$\Phi_{tD}^{(\chi)}=\sum_{\alpha\in (\Z/3\Z)^n}\chi(\alpha)\frac{L_S(\overline{\psi}_{tD_\alpha},1)}{\Omega_t}$$
for any character $\chi: (\Z/3\Z)^n\to \C^\times$. Then the arguments for Corollary \ref{Chi} give

\begin{cor}\label{Chi2} For any cube-free odd integer $D>1$ prime to $3$ and for any character $\chi: (\Z/3\Z)^n\to \C^\times$, we have
$$\ord_2\left(\Phi^{(\chi)}_{D}\right)\geq 0, \quad \quad\ord_2\left(\Phi^{(\chi)}_{2D}\right)\geq -\frac{2}{3} \quad \quad \text{and} \quad \quad \ord_2\left(\Phi^{(\chi)}_{4D}\right)\geq -\frac{1}{3}.$$
\end{cor}

We are now ready to prove:

\begin{thm}\label{integer}For any cube-free integer $N>2$ prime to $3$, we have
$$\ord_2\left(L^{(\mathrm{alg})}(C_{N}, 1)\right)\geq 0.$$
In particular, $L^{(\mathrm{alg})}(C_{N}, 1)$ is an integer. 
\end{thm}

\begin{proof}
 We write $N=tD$ where $2\nmid D$. We can again prove this by induction on the number of distinct prime factors $k(D)$ of $D$. First let $k(D)=1$, so that $D=p$ or $p^2$ for a prime $p$. Let $\chi$ be a character of $\Z/3\Z$ taking $1$ to $\omega$. Then

$$\Phi^{(\chi)}_{tp}-\omega^2 \Phi_{tp}=(1-\omega)\frac{L_S(\ovl{\psi}_{t},1)}{\Omega_t}+(\omega-\omega^2)\frac{L(\ovl{\psi}_{tp},1)}{\Omega_t},$$
where 
$$\frac{L_S(\ovl{\psi}_{t},1)}{\Omega_t}= \left\{ \begin{array}{ll}
\left(\frac{p^2+p}{p^2}\right)\frac{L(\overline{\psi}_{t},1)}{\Omega_t} & \mbox{if $p$ is inert in $K$}\\
\left(\frac{\pi-\left(\frac{t}{\pi}\right)_3}{\pi}\right)\left(\frac{\ovl{\pi}-\left(\frac{t}{\ovl{\pi}}\right)_3}{\ovl{\pi}}\right)\frac{L(\ovl{\psi}_{t},1)}{\Omega_t} & \mbox{if $p\o=(\pi)(\ovl{\pi})$ splits in $K$.}\\
       \end{array} \right.
$$
Recall that  $\frac{L(\overline{\psi}_{1},1)}{\Omega_2}=L^{(\mathrm{alg})}(C_1,1)=\frac{1}{3}$, $\frac{L(\overline{\psi}_{2},1)}{\Omega_2}=L^{(\mathrm{alg})}(C_2,1)=\frac{1}{2}$ and $\frac{L(\overline{\psi}_{4},1)}{\Omega_4}=L^{(\mathrm{alg})}(C_4,1)=1$. Note that in the case when $t=2$ and $p$ is split, we have $\left(\frac{2}{\pi}\right)_3= \left(\frac{\pi}{2}\right)_3\equiv \pi^{\frac{4-1}{3}} \equiv \pi  \bmod 2$, where in the first equality we used the fact that $2$ and $\pi$ are congruent to $\pm1 \bmod 3\o$ and the cubic reciprocity law. It follows that $\ord_2(\pi-\left(\frac{2}{\pi}\right)_3)=\ord_2(\ovl{\pi}-\left(\frac{2}{\ovl{\pi}}\right)_3)\geq 1$. Furthermore, $\ord_2(1-\omega)=0$. Thus, in all cases we have $\ord_2\left(\frac{L(\ovl{\psi}_{tp},1)}{\Omega_t}\right)\geq \ord_2(\Phi_{tp})$. Since $L^{(\mathrm{alg})}(C_{tp},1) \in \Q$, its $2$-adic valuation must be an integer. Thus combining with Theorem \ref{thmPhi2}, we see that $\ord_2\left(L^{(\mathrm{alg})}(C_{tp},1)\right)\geq 0$, as required. The same holds for $L^{(\mathrm{alg})}(C_{tp^2},1)$. The rest of the induction argument follows easily from the proof of Theorem \ref{main1}.
\end{proof}

\begin{remark} Once again, in the case when $k(D)=1$, $t\in \{2,4\}$ and $D\equiv 2,4,5,7\bmod 9$, the proof can be shortened significantly, since by a root number consideration (see Corollary \ref{root}), one of the two primitive Hecke $L$-values vanishes in the sum $\Phi_{tp}$.
\end{remark}

\begin{remark} Lemma \ref{tam} shows that the product of Tamagawa numbers of $C_N$ satisfies
\begin{align*}\ord_2\left(\prod_{q \mathrm{ bad}}c_q\right) = \left\{ \begin{array}{ll}
1&\mbox{if $N\equiv 2,7 \bmod 9$} \\
 0&\mbox{otherwise.}
       \end{array} \right.
 \end{align*}
Assume $L^{(\mathrm{\mathrm{alg}})}(C_{N},1)\neq 0$. Since we know by Cassels' theorem \cite{cassels62} that the order of $\Sha(C_N)$ is a perfect square when it is finite, Conjecture \ref{pBSD} predicts that $\ord_2\left(L^{(\mathrm{\mathrm{alg}})}(C_{N},1)\right)$ is odd whenever $N\equiv 2,7\bmod 9$. Note that by Corollary \ref{corro}, we know that $\ord_2(\mathcal{S}_N)=\ord_2\left(\frac{L^{(\mathrm{\mathrm{alg}})}(C_{N},1)}{\prod_{q \mathrm{ bad}}c_q}\right)$ is even when $N$ is a product of split primes in $K$, but in this case we have $\mathcal{S}_N\neq 0$ only if $N\equiv 1\bmod 9$.
\end{remark}

\newpage

\begin{appendix}

\section{Tamagawa number and root number computations}\label{appendixA}

We will compute the Tamagawa numbers and the root number of the curve
\begin{align*}C_{N}&: x^3+y^3=N,
\end{align*}
for a cube-free integer $N>1$ with $(N,3)=1$. We will first compute the Tamagawa numbers $c_q$ at the primes $q$ of bad reduction for $C_{N}$, so that $q\mid 3N$. The following result is known from \cite{zk87}, but it is given here for the convenience of the reader.

\begin{lem}\label{tam} Let $N>1$ be a cube-free integer prime to $3$. The Tamagawa numbers of $C_{N}$  are given by $c_2=1$,
$$
c_3=
 \left\{ \begin{array}{ll}
1 & \mbox{for $C_{N}$ when $N \equiv 4,5 \bmod 9$}\\
2 & \mbox{for $C_{N}$ when $N\equiv 2,7 \bmod 9$}\\
3 &\mbox{for $C_{N}$ when $N \equiv 1, 8 \bmod 9$}\\
       \end{array} \right.
$$
and for an odd prime $p$ dividing $N$,
$$
c_p=
 \left\{ \begin{array}{ll}
1 &\mbox{if $p\equiv 2\bmod 3$}\\
3 & \mbox{if $p\equiv 1 \bmod 3$. }\\
       \end{array} \right.
$$

\end{lem}

\begin{proof}
We follow Tate's algorithm \cite{tate} and use the usual notation. We will first work with the model $y^2=x^3-2^4 3^3 N^{2}$ for $C_{N}$.
\begin{itemize}
\item $\underline{\text{Tamagawa factor } c_p}$: The type is IV if $p^2\nmid N$, IV$^*$ if $p^2\mid N$. If $p\equiv 2\bmod 3$, then $\left(\frac{-3}{p}\right)=-1$, and the equation $T^2+3^3$ has no roots in $\mathbb{F}_p$. It follows that $c_p=1$. If $p\equiv 1\bmod 3$, then $\left(\frac{-3}{p}\right)=1$, and the equation $T^2+3^3$ has roots in $\mathbb{F}_p$, thus $c_p=3$. 
\item $\underline{\text{Tamagawa factor } c_3}$: We have $3^3\mid a_6$ and $P(T)=T^3-2^4 N^{2}$ has triple roots in $\mathbb{F}_3$ since $P'(T)=3 T^2\equiv 0 \bmod 3$. After the change of variables $x=X-6$ and $y=Y$, the triple root is equal to $0$. The new equation is $Y^2=X^3-18X^2+108X-2^3 3^3(1+2N^{2})$. The equation $T^2-\frac{-2^3 3^3(1+2N^{2})}{3^4}\equiv T^2+1 \bmod 3$ when $N\equiv 4,5\bmod 9$, and this has no roots in $\mathbb{F}_3$, thus $c_3=1$. The equation is congruent to $T^2+2 \bmod 3$ when $N\equiv 1,8\bmod 9$, and has distinct roots in $\mathbb{F}_3$. Hence the type is IV$^*$ and $c_3=3$. This has a double root when $N\equiv 2,7\bmod 9$, but $108/3^3\not\equiv 0\bmod 3$, thus the type is III$^{*}$ and $c_3=2$ in this case.\
\item $\underline{\text{Tamagawa factor } c_2}$:  If $2\nmid N$, then $C_N$ has good reduction at $2$, so clearly $c_2=1$. Now assume $N=2^iD$ where $i=1$ or $2$ and $2\nmid D$. We will deal with the cases $i=1$ and $i=2$ separately.\\
(Case $i=1$) Making the change of variables $x=X$ and $y=Y+3D$, we obtain the equation $Y^2+6DY=X^3-2^2 3^2 D^{2}$. Then $a_3=6D$, $a_6=-2^2 3^2 D^{2}$ and $2^3\nmid b_6=(6D)^2+4a_6$, hence the type is IV. Now, $T^2+3DT+3^2 D^{2}\equiv T^2+T+1$ has no roots in $\mathbb{F}_2$, so $c_2=1$.\\
(Case $i=2$) Making the change of variables $x=X$ and $y=Y+2$,                                                                                                                                                                                                                                                                                                                                                                                                                                                                                                                                                                                                                                                                                                                                                                                                                                                                                                                                                                                                                we obtain the equation $Y^2+4Y=X^3-2^2(3^2D^{2}+1)$. Now, the equation $T^2+\frac{4}{2^2}T+\frac{2^2(3^2D^{2}+1)}{2^4}\equiv T^2+T+1\bmod 2$ has no root in $\mathbb{F}_2$. Thus the type is IV$^*$ and $c_2=1$.
\end{itemize}
\end{proof}
Now, we compute the root number of the curve $C_{N}$. The global root number $\epsilon(C_N/\Q)$ of $C_N$ is given by
\begin{align*}\epsilon(C_N/\Q) &=\prod_{q\leq \infty}\epsilon_q(C_N/\Q)\\
&=-\prod_{q\mid 3N}\epsilon_q(C_N/\Q),
\end{align*}
where $\epsilon_q(C_N/\Q)$ denotes the local root number at a prime $q$, and we always have $\epsilon_\infty(C_N/\Q)=-1$. We write $N=tD$ with $t\in \{1,2,4\}$ and $(2,D)=1$, and let $E_{tD}$ be as given by \eqref{eqE}.

\begin{prop}\label{lroot} The local root numbers of $C_N$ are given by
$$\epsilon_\infty(C_{N}/\Q)=-1,$$
$$\epsilon_2(C_{N}/\Q)= \left\{ \begin{array}{ll}
 -1 &\mbox{if $2\mid N$} \\
 +1 &\mbox{otherwise},
       \end{array} \right.$$
$$\epsilon_3(C_{N}/\Q)= \left\{ \begin{array}{ll}
 -1 &\mbox{if $N\equiv 1,8 \bmod 9$} \\
 +1 &\mbox{otherwise},
       \end{array} \right.$$
and for any odd prime factor $p$ of $N$,
$$\epsilon_{p}(C_{N}/\Q)= \left(\frac{-3}{p}\right).$$
\end{prop}

\begin{proof}
The local root numbers at $2$ and $3$ can be obtained from \cite[Lemma 4.1]{VA}. We have $\epsilon_2(E_{t D}/\Q) = -1$ if $2\mid t$ or $t=1$ and $D^2\equiv 3\bmod 4$. But the latter case does not happen. The claim on the local root number at $3$ follows on noting that $\epsilon_3(E_D/\Q) = -1$ if and only if $-2^4D^2\equiv 2$ or $4\bmod 9$, $\epsilon_3(E_{2D}/\Q) = -1$ if and only if $-D^2\equiv 2$ or $4\bmod 9$ and $\epsilon_3(E_{4D}/\Q) = -1$ if and only if $-2^2 D^2\equiv 2$ or $4\bmod 9$. Since the discriminants of these curves have $p$-adic valuation equal to $4$ or $8$, we know by a result of Rohrlich \cite[Proposition 2]{roh} that $\epsilon_p(E_{tD}/\Q)=\left(\frac{-3}{p}\right)$. 
\end{proof}

\begin{cor}\label{root}Let $D_\alpha =p_1^{\alpha_1}\cdots p_n^{\alpha_n}\equiv 2,4,5,7\bmod 9$ be a cube-free product of primes $p_1,\ldots p_n$, where we may consider $\alpha=(\alpha_1,\ldots \alpha_n)\in (\Z/3\Z)^n$. Then for $t\in \{2,4\}$, we have
$$\epsilon(E_{t D_\alpha}/\Q)=-\epsilon(E_{t D_{2\alpha}}/\Q).$$
In particular, we have at least one of $L(E_{t D_\alpha}/\Q,s)$ and $L(E_{t D_{2\alpha}}/\Q,s)$ vanishing at $s=1$.
\end{cor}

\begin{proof} This follows from Proposition \ref{lroot} on noting that $a^3\equiv \pm1 \bmod 9$ for $a\in \{2,4,5,7\}$.
\end{proof}

\section{Numerical examples}\label{appendixB}

Given a cube-free positive integer $N$ with $(N,3)=1$, $k(N)$ denotes the number of distinct prime factors of $N$. We have $k(N)=r(N)+s(N)$, where $r(N)$ and $s(N)$ denote the number of distinct inert and split prime factors of $N$, respectively. The following computations were obtained using Magma \cite{magma} and Sage \cite{sage}, assuming the generalised Riemann hypothesis whenever $k(N)>3$. We give three tables, one for each curve $C_N$ where $N$ is of the form $D$, $2D$ or $4D$ and $(D,2)=1$. The numerical data is ordered first by $k(N)$ (which affects the $3$-adic valuation of $L^{(\mathrm{alg})}(C_N,1)$ by Theorem \ref{main1}), then by $r(N)$ (which affects the order of $\Sha(C_N)[3]$ by Theorem \ref{3descent})  and finally by the order of the prime factors. 

\begin{equation}\nonumber
\begin{array}{lcclc}
N & (r(N),s(N)) &  N \bmod 9 &  L^{(\mathrm{\mathrm{alg}})}(C_{N},1)& \Sha(C_{N})[3]\\
\hline
19^2 & (0,1) & 1 & 3^2& \text{ trivial}\\  
2683^2& (0,1) &  1 & 3^2 &\text{ trivial}\\  
5 & (1,0) &  5 & 1& \text{ trivial}\\  
5^2 & (1,0) & 7 & 2& \text{ trivial}\\  
19^2\cdot 37^2 & (0,2) & 1 &  3^3& \text{ trivial}\\ 
7\cdot 53 & (1,1) & 2 & 2\cdot 3 &  \text{ trivial}\\  
17^2 \cdot 53 & (2,0) & 1 & 3^3   &  (\Z/3\Z)^2\\ 
17^2 \cdot 53^2 & (2,0) & 8 & 3^3   &  (\Z/3\Z)^2\\  
53^2 \cdot 71^2 & (2,0) & 1 & 3^3   &  (\Z/3\Z)^2\\ 
7^2 \cdot 139^2\cdot 379^2 & (0,3) & 1 &2^6 \cdot 3^4 &  \text{ trivial}\\   
7 \cdot 139 \cdot 397 & (0,3)& 1&3^4 &  \text{ trivial}\\   
7^2 \cdot 139^2 \cdot 397^2 & (0,3)&1& 3^4\cdot 7^2  &  \text{ trivial}\\   
5^2\cdot 7^2 \cdot 13^2 & (1,2)& 7 & 2\cdot 3^2 & \text{ trivial}\\ 
11^2\cdot 29 \cdot 47 & (3,0)& 7 & 2^5 \cdot 3^2 & (\Z/3\Z)^2\\ 
7\cdot 17 \cdot 31 \cdot 53 & (2,2)& 1 & 3^3 &  \text{ trivial}\\ 
5\cdot 11\cdot 17\cdot 53 & (4,0) & 1 & 3^3 &  (\Z/3\Z)^2\\ 
17\cdot 53\cdot 89\cdot 179 & (4,0) & 1 & 3^5 &  (\Z/3\Z)^4\\ 
5^2 \cdot 7^2 \cdot 11^2 \cdot 13^2 \cdot 23 &(3,2) & 5 & 3^6 & (\Z/3\Z)^2\\ 
5^2 \cdot 7^2 \cdot 11^2 \cdot 13^2 \cdot 23^2 &(3,2) & 7 & 2^7\cdot 3^4  & (\Z/3\Z)^2\\ 
5\cdot 17\cdot 23\cdot 29\cdot 41 & (5,0) & 5 & 3^4 & (\Z/3\Z)^4\\
5^2 \cdot 7^2 \cdot 11^2 \cdot 13^2 \cdot 23^2 \cdot 29^2 &(4,2) & 1 & 3^5 & (\Z/3\Z)^2\\ 
\hline
\hline

2\cdot 7 & (1,1) & 5 & 3 & \text{ trivial}\\
2\cdot 7^2 \cdot 13^2 &(1,2) &2 &2 \cdot 3^2 &  \text{ trivial}\\
2\cdot 5^2 \cdot 13^2 &(2,1)&8& 2^2 \cdot 3^2 &  \text{ trivial}\\
2\cdot 5\cdot 109  &(2,1)&1  &3^2 &  \text{ trivial}\\
2\cdot 5\cdot 163&(2,1)&1&3^2&\text{ trivial}\\
2\cdot 5\cdot 181  &(2,1)&1  &3^2 &  \text{ trivial}\\
2\cdot 5 \cdot 1657 & (2,1) & 1& 3^4 &   (\Z/3\Z)^2\\
2\cdot 5^2 \cdot 11^2 &(3,0)&2 & 2\cdot 3^2 &  (\Z/3\Z)^2\\ 
2\cdot 5 \cdot 569& (3,0)&2&2\cdot 3^2& (\Z/3\Z)^2  \\
2\cdot 5 \cdot 569^2& (3,0)&4&2^2\cdot 3^2& (\Z/3\Z)^2\\
2\cdot 5^2 \cdot 569^2& (3,0)&2&2^3 \cdot 3^4&(\Z/3\Z)^2\\
2\cdot 17^2\cdot 23^2 & (3,0)& 5 &2^4\cdot 3^2 &  (\Z/3\Z)^2\\
2\cdot 7 \cdot 139 \cdot 379 & (1,3)&2&2\cdot 3^3 &  \text{ trivial}\\   
2\cdot 7^2 \cdot 139^2 \cdot 379^2 & (1,3)& 2&2\cdot 3^3 &  \text{ trivial}\\   
2\cdot 7\cdot 17 \cdot 23 &(3,1)&2 & 2\cdot 3^3 &  (\Z/3\Z)^2\\ 
2\cdot 23\cdot 29 \cdot 41 &(4,0)& 1 & 2^2\cdot 3^3&  (\Z/3\Z)^2\\
2\cdot 13\cdot 19 \cdot 37\cdot 43 & (1,4) & 2 & 2\cdot 3^4  &\text{ trivial}\\   
2\cdot 5^2 \cdot 7^2 \cdot 11^2 \cdot 13^2 &(3,2) &2 & 2\cdot 3^6 & (\Z/3\Z)^2\\
2\cdot 17\cdot 23 \cdot 29\cdot 53 & (5,0) & 2 &  2\cdot 3^4 & (\Z/3\Z)^4\\
2\cdot 5\cdot 11\cdot 23\cdot 29\cdot 41 & (6,0) & 1 &  3^5 & (\Z/3\Z)^4\\
2\cdot  5^2\cdot7^2\cdot11^2\cdot13^2\cdot17^2\cdot23^2  &(5,2) &5 &   3^8 & (\Z/3\Z)^4\\
\hline
\hline
2^2\cdot 7^2 & (1,1) & 7& 2\cdot 3 &  \text{ trivial}\\
2^2\cdot5^2  &(2,0)&1& 3  &  \text{ trivial}\\
2^2\cdot7 \cdot 13 &(1,2)&4 & 3^2 &  \text{ trivial}\\
2^2\cdot37 \cdot 73 &(1,2) & 4& 2^2\cdot 3^2 &\text{ trivial}\\
2^2\cdot5 \cdot 31 &(2,1)&8& 3^2 &  \text{ trivial}\\
2^2\cdot5^2 \cdot 11^2 &(3,0)&4& 3^4 & (\Z/3\Z)^2\\ 
\end{array}
\end{equation}

\begin{equation}\nonumber
\begin{array}{lcclc}
2^2\cdot5 \cdot 89^2 &(3,0)&2 & 2^5\cdot 3^2 & (\Z/3\Z)^2\\ 
2^2\cdot5^2\cdot 137^2  &(3,0)&4& 3^6  & (\Z/3\Z)^2\\
2^2\cdot47 \cdot 191 &(3,0) & 7 & 2\cdot 3^2 & (\Z/3\Z)^2\\ 
2^2\cdot7^2 \cdot 19^2 \cdot 37^2 & (1,3)& 7& 2^5 \cdot 3^3 & \text{ trivial}\\ 
2^2\cdot11 \cdot 31 \cdot 43 & (2,2)& 8& 3^3 & \text{ trivial}\\ 
2^2\cdot5^2 \cdot 13^2 \cdot 17^2 &(3,1) & 7 & 2^5\cdot 3^3  & (\Z/3\Z)^2\\ 
2^2\cdot5^2 \cdot 11 \cdot 23 &(4,0) & 1 & 3^3  & (\Z/3\Z)^2\\ 
2^2\cdot7 \cdot 13 \cdot 19 \cdot 31 & (1,4)& 7 &2\cdot 3^4 & \text{ trivial}\\ 
2^2\cdot7\cdot 11\cdot 17\cdot 53 & (3,2) & 4 & 3^4 &  (\Z/3\Z)^2\\ 
2^2\cdot5^2\cdot 7^2\cdot 11\cdot 17^2 & (4,1) & 8 & 3^4 &  (\Z/3\Z)^2\\ 
2^2\cdot5\cdot 11\cdot 17^2\cdot 41 & (5,0) & 2 & 2\cdot 3^4 &  (\Z/3\Z)^4\\ 
2^2\cdot5^2 \cdot 7^2 \cdot 11^2 \cdot 13^2 \cdot 23^2 &(4,2) & 1& 3^7  & (\Z/3\Z)^2\\ 
2^2\cdot5\cdot 11\cdot 17\cdot 23\cdot 41 & (6,0) & 8 & 3^5 & (\Z/3\Z)^4\\ 
2^2\cdot5\cdot 17\cdot 23\cdot 29\cdot 41 & (6,0) & 8 & 3^5 & (\Z/3\Z)^4\\ 
2^2\cdot5^2 \cdot 7^2 \cdot 11^2 \cdot 13^2 \cdot 17^2 \cdot 29^2 &(5,2) & 7& 2\cdot 3^8 & (\Z/3\Z)^4\\ 
 \hline
\end{array}
\end{equation}

\end{appendix}

\vspace{20pt}

\begin{center}
\begin{tabular}{@{}p{2.5in}p{}}
Yukako Kezuka \\
Max-Planck-Institut f\"{u}r Mathematik \\
Vivatsgasse 7\\
53111 Bonn\\
Germany\\
\emph{ykezuka@mpim-bonn.mpg.de}\\
\end{tabular}
\end{center}

\end{document}